\documentclass{article}
\usepackage{amsfonts, amsmath, amssymb, amscd, amsthm, graphicx,enumerate,comment}
\usepackage{hyperref,xypic}
\usepackage[usenames]{color}
\usepackage{tikz}
\usepackage{calc}
\usepackage{braket}
\usepackage[marginpar=3cm]{geometry}

\usepackage{thmtools}
\usepackage{thm-restate}

\makeatletter
\def\blfootnote{\xdef\@thefnmark{}\@footnotetext}
\makeatother

\newtheorem{thm}{Theorem}[section]
\newtheorem{cor}[thm]{Corollary}
\newtheorem{lem}[thm]{Lemma}
\newtheorem{prop}[thm]{Proposition}

\theoremstyle{definition}
\newtheorem{defn}[thm]{Definition}
\theoremstyle{remark}
\newtheorem{rem}[thm]{Remark}
\newtheorem{ex}[thm]{Example}

\newtheorem{claim}[thm]{Claim}

\newfont{\eufm}{eufm10}

\renewcommand{\phi}{\varphi}

\newcommand{\R}{\mathbb R}
\newcommand{\N}{\mathbb N}
\newcommand{\Z}{\mathbb Z}
\newcommand{\Q}{\mathbb Q}
\renewcommand{\H}{\mathbb{H}}
\newcommand{\Hl}{\mathcal{H}}

\newcommand{\semi}{\rtimes}

\newcommand{\Ga}{\Gamma}
\newcommand{\acts}{\curvearrowright}
\newcommand{\op}{\operatorname}
\newcommand{\Aut}{\op{Aut}}

\renewcommand{\d }{{\rm d} }
\newcommand{\HG}{\mathcal{H}(G)}
\newcommand{\GG}{\mathcal{G}(G)}
\newcommand{\Lab}{\operatorname{Lab}}
\newcommand{\mc }{\mathcal}
\newcommand{\constone}{D}
\newcommand{\consttwo}{E}
\newcommand{\constthree}{F}
\newcommand{\constfour}{L}
\newcommand{\constfive}{N}
\newcommand{\llangle}{\langle\hspace{-.7mm}\langle }
\newcommand{\rrangle}{\rangle\hspace{-.7mm}\rangle }

\begin{document}
\title{Higher rank confining subsets and hyperbolic actions of solvable groups}
\author{Carolyn R. Abbott \and Sahana Balasubramanya  \and Alexander J. Rasmussen }
\date{}

\maketitle

\begin{abstract}
    Recent papers of the authors have completely described the hyperbolic actions of several families of classically studied solvable groups. A key tool for these investigations is the machinery of \emph{confining subsets} of Caprace, Cornulier, Monod, and Tessera, which applies, in particular, to solvable groups with virtually cyclic abelianizations. In this paper, we extend this machinery and give a correspondence between the hyperbolic actions of certain solvable groups with higher rank abelianizations and confining subsets of these more general groups. We then apply this extension to give a complete description of the hyperbolic actions of generalized solvable Baumslag-Solitar groups and to reprove a result of Sgobbi-Wong computing their Bieri-Neumann-Strebel invariants.
\end{abstract}

\section{Introduction}

Hyperbolic metric spaces are ubiquitous in geometric group theory. They are the key tools for studying small cancellation groups, mapping class groups, right-angled Artin groups, and numerous other groups. In light of this, it is a natural problem to try to give a complete description of all of the (isometric) actions of a given group on hyperbolic metric spaces.  We call such actions \emph{hyperbolic actions}.

In general, this problem is too lofty.  For instance, the trivial actions of a group on all possible metric spaces yields a collection of non-conjugate actions.  To make the problem tractable, in this paper we focus only on \emph{cobounded} actions, that is, actions such that the quotient of the space by the action has finite diameter (see Section \ref{section:background} for the precise definition).  In addition to the examples above, this rules out, for example, all parabolic actions of the group, including those that can be constructed using the Groves--Manning horoball machinery \cite{GrovesManning}.

For many groups studied by geometric group theorists, even the goal of  describing all cobounded hyperbolic actions is  still too lofty. For instance, all acylindrically hyperbolic groups (including most of the groups mentioned in the first paragraph and many others) admit uncountably many cobounded hyperbolic actions that are, in a natural sense, inequivalent \cite{ABO} (see Section \ref{section:background} for the definition of equivalent actions). Nonetheless, significant progress has recently been made in describing the cobounded hyperbolic actions of families of classically studied \emph{solvable} groups. The second author initiated this study by giving a complete description of the cobounded hyperbolic actions of the lamplighter groups $(\Z/n\Z)\wr \Z$ for $n \geq 2$ in \cite{Bal}. The first and third authors then completely described the hyperbolic actions of Anosov mapping torus groups in \cite{Largest} and the hyperbolic actions of solvable Baumslag-Solitar groups in \cite{AR}.

These descriptions also provide additional information about the relationships between the various actions. The first two authors and Osin showed in \cite{ABO} that the set $\mathcal{H}(G)$ of equivalence classes of cobounded hyperbolic actions of a group $G$ admits a partial order, which roughly corresponds  to collapsing equivariant families of subspaces to obtain one hyperbolic action from another (see Section \ref{section:background} for the precise definition). 
Each of the papers \cite{Bal}, \cite{AR}, and \cite{Largest} gives  a complete description of the poset $\mathcal{H}(G)$ for the group $G$ in question.

Aside from the inherent interest of classifying hyperbolic actions, the problem is also connected to the computation of \emph{Bieri-Neumann-Strebel (BNS) invariants} of finitely generated groups (see \cite{Brown} and Section \ref{section:BNS}), and hence to related notions such as property $R_\infty$  and topological fixed point theory 
(see \cite{TabackWong} and \cite{SSV}). For solvable groups, classifying hyperbolic actions is philosophically in line with a body of work centered around understanding the rigidity of continuous actions of solvable groups on manifolds that has a long history; see, e.g., \cite{BBRT,BMNR}, which address an old problem of Plante \cite{Plante}.

The starting point for the techniques developed in \cite{Bal, AR, Largest} is a theory developed by Caprace, Cornulier, Monod, and Tessera in \cite{Amen} that describes cobounded hyperbolic actions of groups with a fixed point on the boundary (so-called \emph{quasi-parabolic actions}) in terms of \emph{confining subsets}. In particular, for a group of the form $G= H\rtimes \Z$, with $\Z$ corresponding to a finite index subgroup of the abelianization, there is a correspondence between confining subsets of $H$ and regular quasi-parabolic actions of $G$. When $G$ is additionally solvable, one may classify hyperbolic actions of $G$ using the confining subsets of $H$. Crucially, the solvable groups discussed in the previous two paragraphs all  have rank one abelianizations, and so the authors were able to apply the work of \cite{Amen} directly. 

However, the machinery of Caprace-Cornulier-Monod-Tessera is not specific enough to classify the hyperbolic actions of solvable groups whose abelianizations have higher rank (see the discussion in Section \ref{section:semiwithZ}). Thus, the techniques developed in \cite{Bal, AR, Largest} do not immediately extend  to such groups.  In this paper we begin the work necessary to extend the theory of Caprace-Cornulier-Monod-Tessera in \cite{Amen} to general finitely generated solvable groups. In particular, we develop a strong definition of confining subsets for semidirect products of the form $H\rtimes \Z^n$, which is sufficient to classify hyperbolic actions of such groups when they are solvable. Such groups arise whenever there is a section of the homomorphism from a group to the free abelian part of its abelianization.

This allows us to  completely describe the cobounded hyperbolic actions of certain solvable groups that were previously out of reach. In contrast to lamplighter groups and solvable Baumslag-Solitar groups, whose posets of cobounded hyperbolic actions are finite, these groups \emph{always} admit uncountably many inequivalent actions on lines, because $\Z^n$ does when $n>1$ (see, for example, \cite[Example~4.23]{ABO}). However, we show that for a certain family of groups, the  remaining cobounded hyperbolic actions can  be understood in a straightforward way.

Our two main theorems give a correspondence between confining proper subsets and quasi-parabolic actions for groups $H\rtimes \Z^n$ and should be compared to \cite[Theorem~4.1]{Amen} in the case $n=1$.   For  solvable groups, every non-elementary hyperbolic action is quasi-parabolic, so this correspondence yields a complete description of the hyperbolic actions in most cases.  The terms in the statements of the following theorems are defined precisely in Sections \ref{section:background} and \ref{section:main}; we give a brief intuitive explanation here.  If $\gamma \colon \Z^n \to \Aut(H)$ is a homomorphism and $G=H\semi_\gamma \Z^n$, then a subset $Q\subseteq H$ is confining under $\gamma$ with respect to a homomorphism $\rho:\Z^n\to \R$ roughly when $H$ is \emph{attracted into} $Q$ under elements of $\Z^n$ with large image under $\rho$ and $Q$ is \emph{nearly closed} under the group operation of $H$. The subset $Z_\rho \subseteq \Z^n$ consists of the elements of $\Z^n$ with small image under $\rho$. The notation $\Gamma(G,S)$ stands for the Cayley graph of $G$ with respect to the (possibly infinite) generating set $S$. A \emph{quasi-parabolic action} is a hyperbolic action with a unique fixed point on the Gromov boundary and infinitely many loxodromic elements. Finally, the \emph{Busemann pseudocharacter} of this action measures the translation of group elements towards or away from the fixed point. 

Our first main result allows us to construct cobounded quasi-parabolic actions from confining subsets for a group $G$ as in the previous paragraph.  In the following statement,  $Z_\rho$ is (roughly) the  set of elements in $\Z^n$ that have small image under $\rho$; see \eqref{eqn:Zrho} in Section~\ref{section:actionconf} for the precise definition.

\begin{thm}\label{thm:main}
Let $G= H\semi_\gamma \Z^n$, where $\gamma\colon \Z^n \to \Aut(H)$ is a fixed homomorphism, and fix a homomorphism $\rho\colon \Z^n\to \R$ and the set $Z_\rho\subseteq \Z^n$ as above. 
If $Q\subseteq H$ is confining under $\gamma$ with respect to $\rho$, then 
\begin{enumerate}[(i)]
\item the Cayley graph $\Ga(G, Q\cup Z_\rho)$ is hyperbolic; 
\item if $Q$ is strictly confining, then $G\curvearrowright \Ga(G,Q\cup Z_\rho)$ is quasi-parabolic, and otherwise this action is lineal; and 
\item the Busemann pseudocharacter for the action $G \acts \Ga(G, Q\cup Z_\rho)$ is proportional to $\rho$ . 
\end{enumerate} 
\end{thm}

Our second main result allows us to recover a strictly confining subset from a cobounded quasi-parabolic action of the group $G$ under certain conditions.

\begin{thm}\label{prop:main}
Let $G = H \rtimes_\gamma \Z^n$, and let $G\curvearrowright X$ be a cobounded quasi-parabolic action on a hyperbolic space $X$.  Let $\beta$ be the Busemann pseudocharacter associated to this action, and assume that $\beta(H) =0$. Then there exists a subset $Q\subseteq H$ which is strictly confining under the action of  $\gamma$ with respect to  $\beta$, such that $X$ is $G$--equivariantly quasi-isometric to $\Gamma(G,Q\cup Z_\beta)$, where $Z_\beta$ is as in \eqref{eqn:Zrho}.
\end{thm}

We note that the assumption that $\beta(H)=0$ is not too restrictive.  In particular, one can check that this holds whenever $H$ is abelian.

To illustrate the use of this theory, we give a complete description of the cobounded hyperbolic actions of  a class of groups related to solvable Baumslag-Solitar groups.  If $k=p_1^{m_1}\cdots p_n^{m_n}$ is the prime factorization of $k$, then the \textit{generalized solvable Baumslag-Solitar group} $G_k$ is defined as  
 $G_k := \Z\left[\frac1k\right] \rtimes_\gamma \Z^n$, where the image under $\gamma$ of the  $i$\textsuperscript{th} generator of $\Z^n$ acts on $\Z[\frac1k]$ by multiplication by $p_i^{m_i}$. Thus such a group has a presentation \[G_k = \left\langle a, t_1,\ldots,t_n \ \big\vert \ [t_i,t_j]=1, t_iat_i^{-1}=a^{p_i^{m_i}} \text{ for all } i,j\right\rangle,\] where $a$ corresponds to a normal generator of the subgroup $\Z[\frac{1}{k}]$.  When $k$ is a power of a prime, the cobounded hyperbolic actions of $G_k$  were classified in \cite{AR}. 

Generalized solvable Baumslag-Solitar groups were first studied in detail by Taback and Whyte, who note that they naturally arise as stabilizers of points at infinity in the  action of $\operatorname{PSL}_2(\Z\left[\frac{1}{k}\right])$ on the product of $\H^2$ with several Bruhat-Tits trees \cite{TabackWhyte} (one associated to $\operatorname{PSL}_2(\mathbb Q_{p_i})$, for each prime $p_i$). Thus, $G_k$ admits an action on $\H^2$ as well as actions on $n$ Bass-Serre trees. The groups $G_k$ have been studied from the perspective of quasi-isometries \cite{TabackWhyte},   twisted conjugacy classes and fixed point theory on compact manifolds \cite{TabackWong}, and  BNS invariants \cite{SgobbiWong,SSV}.

 \begin{restatable}{thm}{Gk} \label{thm:Z1k}
 For any $k\geq 2$ which is not a power of a prime, the poset $\mathcal H(G_k)$ has the following structure: $\mathcal H_{qp}(G_k)$, the subposet of quasi-parabolic actions, consists of $n+1$ incomparable elements. Each quasi-parabolic action dominates a single lineal action; there are uncountably many lineal actions; and all lineal actions dominate a single elliptic action (see Figure \ref{fig:Z[1/k]Structure}). Moreover, every element of $\mc H(G_k)$ contains either an action on a tree, the hyperbolic plane, or a point.
 \end{restatable}

\begin{figure}[h]
\centering
\def\svgscale{0.6}
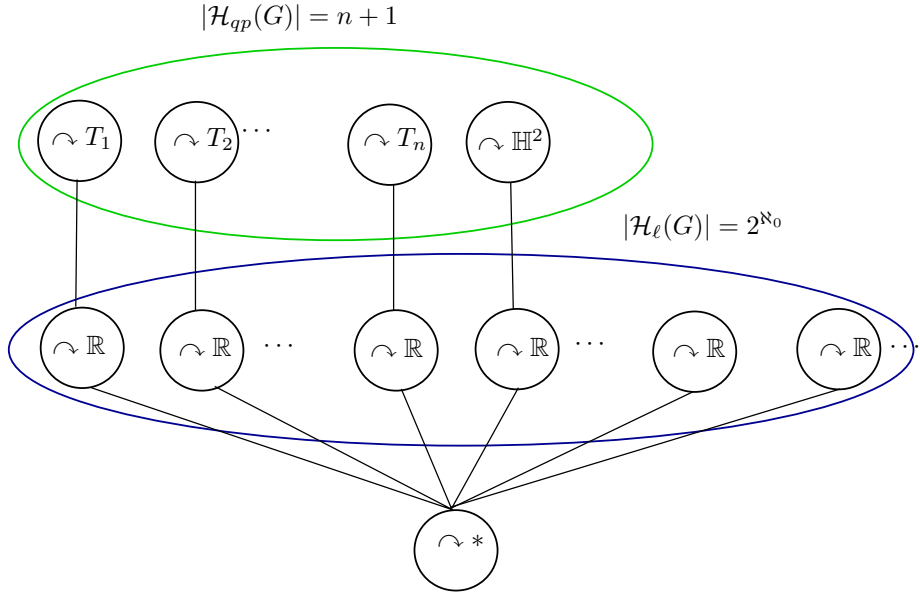
\caption{The poset of hyperbolic structures on the group $G_k=\Z[\frac1k]\rtimes_\gamma \Z^n$.}
\label{fig:Z[1/k]Structure}

\end{figure} 

Theorem \ref{thm:Z1k} reveals that the natural actions of $\operatorname{PSL}_2\left(\Z\left[\frac{1}{k}\right]\right)$ on the hyperbolic plane and Bruhat-Tits trees described give rise to all of the hyperbolic actions of $G_k$, except for actions on points and lines (see Section \ref{section:actiongeometry} and, in particular, Figure \ref{fig:bstrees}). Using a characterization of the BNS invariant of a finitely generated group in terms of its actions on trees due to Brown \cite{Brown}, we use Theorem \ref{thm:Z1k} to compute the BNS invariants of the groups $G_k$, recovering a result of Sgobbi and Wong \cite{SgobbiWong};  see Section \ref{section:BNS}. Taback--Whyte classified the groups $G_k$ up to quasi-isometry, and, in particular, by \cite[Theorem~1.4]{TabackWhyte} the following corollary is immediate.

 \begin{cor}
If $G_k$ and $G_\ell$ are quasi-isometric, then the posets $\mathcal H(G_k)$ and $\mathcal H(G_\ell)$ are isomorphic.
 \end{cor}
 
 The proof of Theorem \ref{thm:Z1k}  reveals that cobounded hyperbolic actions of solvable groups with higher rank abelianizations may sometimes be reduced to the rank one case. Specifically, to prove Theorem \ref{thm:Z1k} we utilize the classification of cobounded hyperbolic actions of the Baumslag-Solitar group $BS(1,k)$ given in \cite{AR}. In light of this and the classification of cobounded hyperbolic actions of wreath products $(\Z/n\Z)\wr \Z$ given in \cite{Bal}, a natural next step would be to apply Theorems \ref{thm:main} and \ref{prop:main} and the techniques developed in \cite{Bal,AR,Largest} to attempt to classify the hyperbolic actions of wreath products $A\wr B$ and extensions $A\rtimes B$ when $A$ and $B$ are finitely generated abelian groups.   
 
Throughout the paper we assume (and rely heavily on the fact that) our groups have a decomposition as a semidirect product.  More generally, one could try to generalize this machinery to groups for which the short exact sequence associated to the map to the abelianization does not split.  For example, it would be interesting to understand if the techniques in this paper could be extended to describe $\mc H(G)$ when $G$ is a finitely generated metabelian group. \\

\noindent{\bf Acknowledgements.} The authors thank the anonymous referee for a careful reading of the paper and useful comments. The first author was partially supported by NSF Awards DMS-1803368 and DMS-2106906.  The second author was supported by the Deutsche Forschungsgemeinschaft (DFG, German Research Foundation) -Project-ID 427320536 – SFB 1442, as well as under Germany's Excellence Strategy EXC 2044 390685587, Mathematics Münster: Dynamics–Geometry–Structure.  The third author was partially supported by NSF grant DMS-1840190.

\section{Background} \label{section:background}

\subsection{Comparing generating sets and group actions} 
 
Throughout this paper, all group actions on metric spaces are assumed to be isometric. Given a metric space $X$, we denote by $\d_X$ the distance function on $X$. If $G$ is a group and $S$ is a generating set of $G$, then $\|\cdot\|_S$ denotes the word norm on $G$ and $d_S$ denotes the word metric $d_S(g,h)=\|gh^{-1}\|_S$.

\begin{defn}[{\cite[Definition 1.1]{ABO}}]\label{def-GG}
Let $S$, $T$ be two (possibly infinite) generating sets of a group $G$. We say that $S$ is \emph{dominated} by $T$, written $S\preceq T$, if the identity map on $G$ induces a Lipschitz map between metric spaces $(G, \d_T)\to (G, \d_S)$. This is equivalent to the requirement that $ \sup_{t\in T}\|t\|_S<\infty$.  The relation $\preceq$ is a preorder on the set of generating sets of $G$, and therefore it induces an equivalence relation in the standard way:
$$
S\sim T \;\; \Leftrightarrow \;\; S\preceq T \; {\rm and}\; T\preceq S.
$$
This is equivalent to the condition that the Cayley graphs $\Ga(G,S)$ and $\Ga(G, T)$ are $G$--equivariantly quasi-isometric. We denote by $[S]$ the equivalence class of a generating set $S$, and by $\GG$ the set of all equivalence classes of generating sets of $G$. The preorder $\preceq$ induces a partial order $\preccurlyeq $ on $\GG$ by the rule
$$
[S]\preccurlyeq [T] \;\; \Leftrightarrow \;\; S\preceq T.
$$
\end{defn}

For example, all finite generating sets of a finitely generated group are equivalent and the equivalence class containing any finite generating set is the largest element of $\GG$. For every group $G$, the smallest element of $\GG$ is $[G]$. Note also that this order is ``inclusion reversing": if $S$ and $T$ are generating sets of $G$ such that $S\subseteq T$, then $T\preceq S$.

To define a hyperbolic structure on a group, we first recall the definition of a hyperbolic space. In this paper we employ the definition of hyperbolicity via the Rips condition. 

\begin{defn} A metric space $X$ is called \emph{$\delta$--hyperbolic} if it is geodesic and for any geodesic triangle $\Delta $ in $X$, each side of $\Delta $ is contained in the union of the closed $\delta$--neighborhoods of the other two sides.
\end{defn} 

\begin{defn}[{\cite[Definition 1.2]{ABO}}]
A \emph{hyperbolic structure} on $G$ is an equivalence class $[S]\in \GG$ such that $\Gamma (G,S)$ is hyperbolic. Since hyperbolicity of a space is a quasi-isometry invariant, this definition is independent of the choice of a particular representative in the equivalence class $[S]$. We denote the set of hyperbolic structures by $\HG$. It is a sub-poset of $\GG$ with the restriction of the partial order on $\GG$.  \end{defn}

The poset $\HG$ classifies the \emph{cobounded} hyperbolic actions of $G$ up to coarsely equivariant quasi-isometry, as we now summarize.
\begin{defn}
An action $G\curvearrowright X$ is \emph{cobounded} if for some (equivalently any) $x\in X$ there exists $R>0$ such that every point of $X$ is distance at most $ R$ from some point of the orbit $Gx$. Given two cobounded hyperbolic actions $G\curvearrowright X$ and $G\curvearrowright Y$, a map $f:X\to Y$ is \emph{coarsely equivariant} if for any $x\in X$ we have \[\sup_{g\in G} d_Y(f(gx),gf(x))<\infty.\] Given $C>0$, the map $f$ is \emph{$C$--coarsely Lipschitz} if \[d_Y(f(x),f(y))\leq Cd_X(x,y)+C\] for all $x,y\in X$. It is a \emph{$C$--quasi-isometry} if it is $C$--coarsely Lipschitz and also satisfies \[\frac{1}{C}d_X(x,y)-C\leq d_Y(f(x),f(y)).\]
\end{defn}

Given  a cobounded hyperbolic action $G\curvearrowright X$, there is an associated hyperbolic structure given by the \emph{Schwarz-Milnor Lemma}:

\begin{lem}[{\cite[Lemma 3.11]{ABO}}]\label{lem:MS}
Let $G\curvearrowright X$ be a cobounded hyperbolic action of $G$. Let $B\subseteq X$ be a bounded subset such that $\displaystyle \bigcup_{g\in G} gB=X$. Let $D=\operatorname{diam}(B)$ and let $x\in B$. Then $G$ is generated by the set \[S=\{g\in G: d_X(x,gx)\leq 2D+1\},\] and $X$ is $G$--equivariantly quasi-isometric to $\Gamma(G,S)$.
\end{lem}

 Thus, up to  equivariant quasi-isometries, hyperbolic actions of $G$ correspond to actions of $G$ on its hyperbolic Cayley graphs. There is an equivalence relation on hyperbolic actions given by $G$--coarsely equivariant quasi-isometry and a preorder on hyperbolic actions given by $G$--coarsely equivariant coarsely Lipschitz maps. With these relations, the set of \emph{equivalence classes of cobounded hyperbolic actions} of $G$ becomes a poset. This poset turns out to be isomorphic to $\HG$. We refer the reader to \cite[Section 3]{ABO} for more details.

\subsection{General classification of hyperbolic actions} 
We now recall some standard facts about groups acting on hyperbolic spaces. For details the reader is referred to \cite{Gro}. Given a hyperbolic space $X$, we denote by $\partial X$ its Gromov boundary with the visual topology. In general, $X$ is not assumed to be proper, and its boundary is defined as the set of equivalence classes of sequences convergent at infinity. Given a group $G$ acting on a hyperbolic space $X$,  we denote by $\Lambda (G)$ the set of limit points of $G$ on $\partial X$. That is, $$\Lambda (G)=\partial X\cap \overline{Gx},$$ where $\overline{Gx}$ denotes the closure of a $G$--orbit in $X\cup \partial X$, for any choice of basepoint $x\in X$.  This definition is independent of the choice of $x\in X$. The action of $G$ is called \emph{elementary} if $|\Lambda (G)|\le 2$ and \emph{non-elementary} otherwise. The action of $G$ on $X$ naturally extends to an action of $G$ on $\partial X$ by homeomorphisms.

\begin{defn} Given an action of a group $G$ on a hyperbolic space $X$, an element $g\in G$ is called
\begin{enumerate}
\item[(i)] \emph{elliptic} if $\langle g\rangle $ has bounded orbits;
\item[(ii)] \emph{loxodromic} if the map $n \mapsto g^nx, n \in \Z$ is a quasi-isometric embedding for some (equivalently any) $x \in X$;
\item[(iii)] \emph{parabolic} otherwise.
\end{enumerate} \end{defn}

Every loxodromic element $g\in G$ has exactly $2$ fixed points $g{^{\pm \infty}}$ on $\partial X$, where $g{^{+\infty}}$ (respectively, $g{^{-\infty}}$) is the limit of the sequence $(g{^{n}}x)_{n\in \mathbb N}$ (respectively, $(g{^{-n}}x)_{n\in \mathbb N}$) for any fixed $x\in X$. Thus $\Lambda (\langle g\rangle) =\{ g{^{\pm \infty}}\}$.

The following theorem summarizes the standard classification of groups acting on hyperbolic spaces due to Gromov \cite[Section 8.2]{Gro} and the results  \cite[Propositions 3.1 and 3.2]{Amen}.

\begin{thm}\label{ClassHypAct}
Let $G$ be a group acting on a hyperbolic space $X$. Then exactly one of the following conditions holds.
\begin{enumerate}
\item[1)] $|\Lambda (G)|=0$. Equivalently,  $G$ has bounded orbits. In this case the action of $G$ is called \emph{elliptic}.

\item[2)] $|\Lambda (G)|=1$. In this case the action of $G$ is called \emph{parabolic}. A parabolic action cannot be cobounded. 

\item[3)] $|\Lambda (G)|=2$. Equivalently, $G$ contains a loxodromic element and any two loxodromic elements have the same limit points on $\partial X$. In this case the action of $G$ is called \emph{lineal}.

\item[4)] $|\Lambda (G)|=\infty$. Then $G$ always contains loxodromic elements. In turn, this case breaks into two subcases.
\begin{enumerate}
\item[(a)] $G$ fixes a point of $\partial X$. Equivalently, any two loxodromic elements of $G$ have a common limit point on the boundary. In this case the action of $G$ is called \emph{quasi-parabolic} or \emph{focal}. 
\item[(b)] $G$ does not fix any point of $\partial X$. In this case the action of $G$ is said to be of \emph{general type}.
\end{enumerate}
\end{enumerate}
\end{thm}

The following classification of hyperbolic structures is an immediate consequence of the above theorem.
\begin{thm}[{\cite[Theorem 4.6]{ABO}}]\label{main00}
For every group $G$, $$\Hl(G)=\Hl_e(G)\sqcup \Hl_{\ell} (G)\sqcup \Hl_{qp} (G)\sqcup \Hl_{gt}(G)$$ 
where the sets of elliptic, lineal, quasi-parabolic, and general type hyperbolic structures on $G$ are denoted by $\Hl_e(G)$, $\Hl_{\ell} (G)$, $\Hl_{qp} (G)$, and $\Hl_{gt}(G)$ respectively. 
\end{thm}

\subsection{The Busemann pseudocharacter} A function $q\colon G\to \mathbb R$ is a \emph{quasi-character} (or \emph{quasi-morphism}) if there exists a constant $D$ such that $$|q(gh)-q(g)-q(h)|\le D$$ for all $g,h\in G$. We say that $q$ has \emph{defect at most $D$}. If, in addition, the restriction of $q$ to every cyclic subgroup of $G$ is a homomorphism, $q$ is called a \emph{pseudocharacter} (or \emph{homogeneous quasi-morphism}). Every quasi-character $q$ gives rise to a pseudocharacter $\beta$ defined by the following limit, which always exists, for any $g\in G$:
$$
\beta(g)=\lim_{n\to \infty} \frac{q(g{^{n}})}n.
$$
The function $\beta$ is called the \emph{homogenization of $q$.} It is straightforward to check that
$$
 |\beta(g) -q(g)|\le D
$$
for all $g\in G$ if $q$ has defect at most $D$.

Given any action of a group on a hyperbolic space fixing a point on the boundary, one can associate the \emph{Busemann pseudocharacter}. We briefly recall the construction and necessary properties here, and refer the reader to \cite[Sec. 7.5.D]{Gro} and \cite[Sec. 4.1]{Man} for further details.

\begin{defn}\label{Bpc}
Let $G$ be a group acting on a hyperbolic space $X$ and fixing a point $\xi\in \partial X$.
Fix any $x\in X$ and let ${\bf x}=(x_i)$ be any sequence of points of $X$ converging to $\xi$. Then  the function $q_{\bf x}\colon G\to \mathbb R$ defined by
$$
q_{\bf x}(g)=\limsup\limits_{n\to \infty}\left(\d_X (gx, x_n)-\d_X(x, x_n)\right)
$$
is a quasi-character. Its homogenization $\beta_{\bf x}$ is called the \emph{Busemann pseudocharacter}. It is known that this definition is independent of the choice of $\bf x$ (see \cite[Lemma 4.6]{Man}), and thus we can drop the subscript in $\beta_{\bf x}$. An element $g\in G$ is loxodromic with respect to the action of $G$ on $X$ if and only if $\beta(g)\ne 0$.  In particular, $\beta$ is not identically zero whenever $G \curvearrowright X$ is quasi-parabolic. If $\beta$ is a homomorphism, then the action $G\curvearrowright X$ is called \emph{regular}. For examples of non-regular actions, we refer the reader to \cite[Example 3.12]{Amen}
\end{defn}

\subsection{Quasi-parabolic structures on $H \rtimes  \Z$}
\label{section:semiwithZ}
Consider a group $G=H\rtimes_\alpha \Z$ where a generator $t\in \Z$ acts on $H$ by conjugation via $tht^{-1}=\alpha(h)$ for any $h\in H$, and $\alpha\colon \mathbb Z\to \Aut(H)$ is a homomorphsim.  
The following definition from \cite[Section~4]{Amen} forms the basis of the work we do in this paper. Here $Q\cdot Q$ denotes the set of elements $\{gh \in H \mid g,h\in Q\}$.

\begin{defn}\label{genconfine} Let $(H, \cdot)$ be a group, $Q$  a symmetric subset of $H$, and  $\alpha$ be an automorphism of $H$. We say that the action of $\alpha$ is \textit{(strictly) confining $H$ into $Q$} if it satisfies the following conditions$\colon$ 
\begin{itemize}
\item[(a)] $\alpha(Q)$ is (strictly) contained in $Q$. 
\item[(b)] $H = \displaystyle \bigcup_{n \geq 0} \hspace{5pt} \alpha^{-n}(Q)$.
\item[(c)] $\alpha^{n_0}(Q \cdot Q) \subseteq Q$ for some non-negative integer $n_0$. 
\end{itemize}
We also call the set $Q$ \emph{confining under $\alpha$}.
\end{defn}

The definition of a confining subset given in \cite{Amen} does not require symmetry of the subset $Q\subseteq H$. However, according to \cite[Theorem~4.1]{Amen}, to classify regular quasi-parabolic structures on a group it suffices to consider only confining subsets which are symmetric. See also \cite[Proposition~2.6]{AR}.

\begin{prop}[{\cite[Proposition 4.6]{Amen}}]\label{niceprop} Let $H$ be a group and  $\alpha$ an automorphism of $H$ which confines $H$
into some subset $Q \subseteq H$. Consider the group $G=H\rtimes_\alpha \Z$, and let $t$ denote a generator of $\Z$. Define $S=Q\cup \{t^{\pm 1}\}\subseteq G$. Then $\Gamma(G,S)$ is Gromov hyperbolic. If $\alpha(Q) \subsetneq Q$, then the action $G\curvearrowright \Gamma(G,S)$ is regular quasi-parabolic. 
\end{prop} 

If the action is confining but not strictly confining, that is, if $\alpha(Q) =Q$, then the above theorem still holds with the difference that $\Gamma(G,S)$ is quasi-isometric to a line; see the discussion after the statement of \cite[Theorem~4.1]{Amen}.   The resulting action is thus lineal.

If $G$ is a group with higher rank abelianization, then \cite[Theorem~4.1]{Amen} also describes regular quasi-parabolic actions of $G$ in terms of confining subsets of subgroups of the form $[G,G] \rtimes \langle \alpha \rangle$, where $\alpha$ is an arbitrary element of $G$.   In practice, the wide family of groups $[G,G]\rtimes \langle \alpha \rangle$ under discussion and the inability to compare them makes it difficult to directly apply the characterization of \cite[Theorem~4.1]{Amen} to classify hyperbolic actions of $G$.  
In Section \ref{section:main}, we develop a strong notion of confining subsets for groups of the form $H\rtimes \Z^n$, which allows us to work with the group itself rather than a collection of subgroups, 
making it significantly easier to classify the hyperbolic actions.   

\section{Proofs of main theorems} \label{section:main}

Let $G= H\semi_\gamma \Z^n$, where $\gamma\colon \Z^n \to \Aut(H)$ is a fixed homomorphism. An element $z\in \Z^n$ acts on $H$ by conjugation via $zhz^{-1}=\gamma(z)(h)$. The following generalizes the definition of a confining subset from \cite{Amen} to all cases when $n \geq 1$.

\begin{defn} \label{def:confining}
Fix a (non-zero) homomorphism $\rho\colon \Z^n \to \R$.  A symmetric subset $Q$ of $H$ is  \emph{confining under $\gamma$ with respect to $\rho$} if the following three conditions hold.
\begin{itemize}
\item[(a)] For all $z\in \Z^n$ with $\rho(z)\geq 0$, $\gamma(z)(Q)\subseteq Q$.
\item[(b)] For each $h\in H$, there exists $z\in \Z^n$ such that $\gamma(z)(h)\in Q$.
\item[(c)] There exists $z_0\in\Z^n$ such that $\gamma(z_0)(Q\cdot Q)\subseteq Q$.
\end{itemize}
If there exists $z\in\Z^n$ such that $\gamma(z)(Q)\subsetneq Q$, then $Q$ is \emph{strictly} confining under $\gamma$ with respect to $\rho$.
\end{defn}

\begin{rem} \label{rem:equivconditions}Condition (b) is equivalent to the following: For any $h \in H$, there exists $R_h \in \R$ such that  $ \gamma(z)(h)  \in Q$ for any $z\in \Z^n$ with $\rho(z) \geq R_h$. Moreover, condition (c) is equivalent to the following: There exists $R_0\in \R$ such that $\gamma(z)(Q\cdot Q)\subseteq Q$ for \emph{any} $z$ with $\rho(z)\geq R_0$. Further, although not stated explicitly,  the condition for $Q$ to be strictly confining must be satisfied by an element $z \in \Z^n$ such that $\rho(z) > 0$. All these facts can be checked using condition (a) and the fact that $\rho$ is a homomorphism. 
\end{rem} 

\begin{rem}
While we choose to work in the context of groups $G=H\rtimes \Z^n$, we could equally well have considered groups $G=H\rtimes A$ where $A$ is an infinite, finitely generated abelian group.  If $A=T\times \Z^n$ for some $n$ where $T$ is torsion, then any group $H\rtimes A$ can be written as $H'\rtimes \Z^n$ where $H'=H\rtimes T$.  Thus we do not lose any generality in assuming $G=H\rtimes \Z^n$ .
\end{rem}

By a Schwarz-Milnor argument, cobounded quasi-parabolic actions of a group $G$ are described by studying the ``coarse stabilizers" of points in the space. When $G=H \rtimes \Z^n$, confining subsets naturally arise by intersecting such coarse stabilizers with $H$.  Given such an action, the homomorphism $\rho$ in Definition~\ref{def:confining} will be the associated Busemann pseudocharacter, which measures the translation of elements towards a fixed point on the boundary.   When $n=1$, consider a quasi-parabolic action on a hyperbolic space, and let $Q\subset H$ be the intersection of the coarse stabilizer of  a point $p$ with $H$.  The intuition behind why $Q$ is confining is as follows (a precise description in the case of a solvable Baumslag-Solitar group acting on $\mathbb H^2$ is in the discussion after \cite[Definition~2.3]{AR}).  If $g\in H$ and $t$ is the generator of $\Z$, we consider the action of a conjugate $\gamma(t)(g)=t^{-k}gt^k$.  First, $t^k$ translates $p$ very far towards the single fixed point on the boundary, then $g$ moves this point a small distance along a horosphere, and then $t^{-k}$ moves this points back to lie on the original horosphere.  The net result is that $t^{-k}gt^k$ moves $p$ a much smaller distance than $g$ itself does, and for large enough $k$, the element $\gamma(t^{-k})(g)=t^{-k}gt^{k}\in Q$.  When $n>1$, the role of $\rho$ becomes apparent.  Any element of $\Z^n$ with positive image under $\rho$ will translate elements towards the single fixed point in the boundary.  Thus the same discussion above holds for \textit{any}  $z\in\Z^n$ with $\rho(z)>0$.

\subsection{Actions from confining subsets}
\label{section:actionconf}
The goal of this section is to prove Theorem \ref{thm:main}.

Fix a generating set $\{t_1,\dots, t_n\}$ of $\Z^n$.  Given a non-zero homomorphism $\rho\colon \Z^n\to \R$, we construct a (possibly infinite) generating set of $\Z^n$ as follows.  Fix a constant $C_\rho>0$  such that $\rho(t_i) \in [-C_\rho, C_\rho]$ for all $i \in \{1,2,\dots,n\}$ and there exists $y\in \Z^n$ such that $\rho(y)=C_\rho$, and let 
\begin{equation}\label{eqn:Zrho}
Z _\rho= \{z \in \Z^n \mid |\rho(z)| \leq C_\rho \}.
\end{equation}

Suppose that $Q\subseteq H$ is confining under $\gamma$ with respect to $\rho$. It is straightforward to check that $Q\cup Z_\rho$ is symmetric, and $Q\cup Z_\rho$ generates $G$ by Definition \ref{def:confining}(b).   We denote the word norm on $G$ with respect to $Q\cup Z_\rho$ by $\|\cdot\|_{Q\cup Z_\rho}$.

Since $\rho$ is a non-zero homomorphism, $Z_\rho$ is a proper subset of $\Z^n$ which generates $\Z^n$ by definition.  By \cite[Lemma 4.15]{ABO}, the Cayley graph $ \Ga(\Z^n, Z_\rho)$ is a quasi-line (that is, quasi-isometric to a line) and the action of $\Z^n $ on $ \Ga(\Z^n, Z_\rho)$ is lineal.  For the rest of the section, we fix a hyperbolicity constant $\delta'$ for $\Ga(\Z^n,Z_\rho)$. 

Our first goal is to prove that $\Gamma(G,Q\cup Z_\rho)$ is hyperbolic, which we will do by understanding what happens to a path when we put its label into a normal form.  Our strategy follows the structure of the proof of \cite[Proposition~4.6]{Amen}, though our more general situation provides additional complications.

We begin by bounding the length of geodesic edge paths in $\Ga(G,Q\cup Z_\rho)$ whose labels are in $Q$. Recall that in a Cayley graph of a group  with respect to a (possibly infinite) generating set, each edge comes with a label that is an element of the generating set.  Thus any path $\gamma$ in the Cayley graph has a label $\Lab(\gamma)$ which is the word formed by reading  off the labels of the edges in the path, in the order they are traversed.     

\begin{lem}[{cf. \cite[Lemma 4.7]{Amen}}] \label{control}There exists a $k_0 \in \mathbb{N}$ such that any geodesic edge path in $\Ga(G, Q\cup Z_\rho)$ each of whose edge labels lies in $Q$ has length at most $k_0$. \end{lem}

\begin{proof} Let $B(r)$ denote the ball of radius $r$ centered at the identity in $\Ga(G, Q\cup Z_\rho)$.  By $B(r)\cap H$, we mean the intersection of the zero skeleton of the ball $B(r)$ with the subgroup $H$.  Thus $B(1) \cap H = Q$.  By Definition \ref{def:confining}(c), we have $Q^2 = Q \cdot Q \subseteq \gamma(z_0)^{-1}(Q)=\gamma(z_0^{-1})(Q)$ for some $z_0\in \Z^n$, and, more generally, 
$$Q^{2^m} \subseteq \gamma(z^{-m}_0)(Q).$$ 

For any $q \in Q$, 
\[
\|\gamma(z_0^{-m})(q)\|_{Q\cup Z_\rho}=\|z_0^{-m}qz_0^m\|_{Q\cup Z_\rho}\leq 2m\|z_0\|_{Q\cup Z_\rho}+1.
\]
Letting $\ell_0=\|z_0\|_{Q\cup Z_\rho}$, we have 
\[
\left( B(1) \cap H \right)^{2^m}=Q^{2^m}\subseteq \gamma(z_0^{-m})(Q)\subseteq B(2m\ell_0 + 2) \cap H.
\]

In particular, it follows that any geodesic edge path such that the label of each edge is in $Q$ that has length at most $ 2^m$ must satisfy $$2^m \leq 2m\ell_0 +2.$$ But then $m$ is bounded by $\kappa= 4 \log_2(\ell_0 +2)$, and so the length is bounded  by $k_0 = 2\kappa \ell_0 +2$.
\end{proof}

\begin{rem}\label{rem:Hnoloxos}
Lemma \ref{control} shows that $H$ cannot contain any loxodromic isometries with respect to the action of $G$ on $\Ga(G,Q\cup Z_\rho)$.   To see this, notice that given any $h\in H$, there is some $z\in \Z^n$ such that $\gamma(z)(h)=zhz^{-1}\in Q$.  In other words, every $h\in H$ is conjugate to an element of $Q$.  As $Q$ is exponentially distorted, it contains no loxodromic elements with respect to this action. Thus neither does $H$.  
\end{rem}

The next lemma provides a preferred form for paths in $\Ga(G,Q\cup Z_\rho)$ and shows that any geodesic is at uniformly bounded Hausdorff distance from a path in this preferred form. This is related to \cite[Lemma 4.8]{Amen}.

\begin{lem} \label{newpaths}
Let $\alpha$ be an edge path of length $m$ in $\Ga(G, Q\cup Z_\rho)$ with exactly $k$ edges whose labels are in $Q$ and such that all subpaths with labels in $\Z^n$ are geodesics.  Then there exists a path $\tau$ with the same endpoints as $\alpha$ of the form 
\[
\tau=\tau_1\tau_2\tau_3
\]
with $\Lab(\tau_1)=a_1a_2\ldots a_s$, $\Lab(\tau_2)=g_1g_2\ldots g_k$, and $\Lab(\tau_3)=b_1b_2\ldots b_r$ that satisfies the following properties:
\begin{itemize}
\item $\tau\subset \mc N_{2(k+1)\delta' +2k}(\alpha)$, where $\delta'$ is the hyperbolicity constant of $\Ga(\Z^n,Z_\rho)$;
\item $\tau_1$ and $\tau_3$ are geodesics;
\item $g_i\in Q$ for all $1\leq i\leq k$;
\item $a_j\in Z_\rho$ satisfies $\rho(a_j)<0$ for each $1\leq j\leq s$; 
\item $b_l\in Z_\rho$ satisfies  $\rho(b_l) \geq 0$ for each $1\leq l\leq r$; and 
\item $s + k + r \leq m$
\end{itemize}

Moreover, if $\alpha$ is a geodesic edge path in $\Ga(G, Q\cup Z_\rho)$, then $\tau$ is also a geodesic edge path, and $\alpha$ and $\tau$ are at Hausdorff distance at most $2(k_0 +1)\delta' + 2k_0$, where $k_0$ is the constant from Lemma~\ref{control}. 
\end{lem}

\begin{proof}  For every $q \in Q$ and $z \in Z_\rho$, either $\gamma(z)(q) \in Q$ (if $\rho(z) \geq 0$) or $\gamma(z^{-1})(q) \in Q$ (if $\rho(z) <0$). Thus when $\rho(z) \geq 0$, we have $zq = \gamma(z)(q)z=q'z$, where $q' =\gamma(z)(q)\in Q$. Similarly, when $\rho(z) < 0$, we may replace $qz$ with $zq''$ where $q'' =\gamma(z)^{-1}(q)=\gamma(z^{-1})(q)\in Q$. If we perform one such operation to the label of a path, the result is a  new path  at distance at most 1 from the original path in $\Ga(G, Q\cup Z_\rho)$ (see Figure \ref{fig1}).   Now suppose we have a word $z_1\ldots z_m\in \Z^n$ with $z_i\in Z_\rho$ and $\rho(z_i)\geq 0$ for each $1\leq i\leq m$.  Let $q\in Q$ and consider a path $\zeta$ with label $z_1\ldots z_m q$.    By performing this operation $m$ times, each time moving one letter in $Z_\rho$ past $q$, we obtain a sequence of paths $z_1\ldots z_{m-i}q_iz_{m-i+1}\ldots z_m$. To see that they are mutually at Hausdorff distance one see Figure \ref{fig:Movingzs}.

\begin{figure}
	\centering
	\def\svgscale{0.6}
	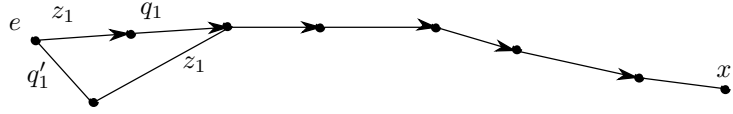
           \caption{Performing a single operation $z_1q_1=q_1'z_1$ results in a path at distance at most one from the original path in $\Ga(G,Q\cup Z_\rho)$.}
          \label{fig1}
\end{figure}

\begin{figure}
	\centering
	\def\svgscale{0.7}
	\small{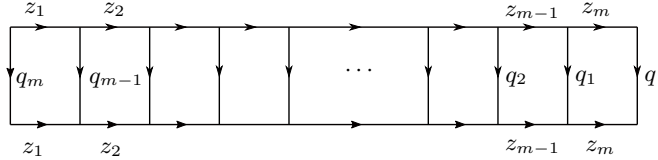}
           \caption{Performing $m$ operations $z_1z_2\ldots z_mq = q_mz_1\ldots z_m$ results in a path at distance at most one from the original path in $\Ga(G,Q\cup Z_\rho)$.}
          \label{fig:Movingzs}
\end{figure}

Let $\Lab(\alpha) = e_1e_2\ldots e_m$ be the label of the given path $\alpha$ in $\Ga(G, Q\cup Z_\rho)$.  The $k$ edges whose labels are in  $Q$ partition this path into at most $k+1$ segments containing edges whose labels are in  $Z_\rho$, each of which is a geodesic segment by assumption, and at most $k$ segments whose edges have labels in $Q$. We may thus think of the path $\alpha$ as having the form 
$$
\alpha=\omega_1\mu_1\omega_2\mu_2\ldots \mu_l\omega_{l+1}
$$
where each $\omega_i$ is a geodesic whose edges all have labels in $Z_\rho$, each $\mu_j$ is a path whose edges all have labels in $Q$, and $l \leq k$.  Let $\Lab(\omega_i)=w_i$ and  $\Lab(\mu_i)=u_i$.

Since $\Z^n$ is abelian, we may reorder the letters in the word $w_1$ so that $w_1=v_1^-v_1^+$, where each letter in $v_1^-$ has negative image under $\rho$ and each letter in $v_1^+$ has non-negative image under $\rho$.  
We replace the geodesic $\omega_1$ with a path (which is also necessarily geodesic) whose label is $v_1^-v_1^+$.  Since $\Ga(\Z^n,Z_\rho)$ is $\delta'$--hyperbolic, this results in a path at Hausdorff distance at most $\delta'$ from $\alpha$ with label 
\[
v^-_1 v^+_1u_1w_2u_2\ldots u_l w_{l+1}.
\]

Let  $u\in Q$ be the first letter in $u_1$.  As described in the first paragraph, we have $v_1^+u=u'v_1^+$,  where $u'=\gamma(v_1^+)(u)$.  We replace the collection of edges in $\alpha_1$ labeled by $v_1^+u$ with a new path labeled by $u'v_1^+$.  If $q$ is the second letter of $u_1$, then we again replace the subpath of this new path labeled by $v_1^+q$ with a path labeled by $q'v_1^+$, where $q'=\gamma(v_1^+)q$.  Continuing in this manner, we may move the subpath labeled by $v^+_{1}$ past all the edges contained in $\mu_1$.  Each step in this process produces a path at Hausdorff distance 1 from the previous path, and thus in the end we have   produced a path $\alpha'$ at Hausdorff distance at most $\delta' + \ell(\mu_1)$   from $\alpha$ with label 
  $$ 
  v^-_1u'_1 v^+_1w_2u_2\ldots u_l w_{l+1},
  $$ 
  for some word $u_1'$, each letter of which is in $Q$.

The subpath of $\alpha'$ labeled by $v^+_1w_2$ is a concatenation of two geodesic paths.  
Let $\nu_2$ be a geodesic in $\Ga(\Z^n,Z_\rho)$  with the same endpoints, 
so that $\Lab(\nu_2)\in \Z^n$.  As before, since $\Z^n$ is abelian, we may rearrange the edges in $\nu_2$ so that $\Lab(\nu_2)=v_2^-v_2^+$, where each letter of $v_2^-$ has negative image under $\rho$ and each letter of $v_2^+$ has non-negative image under $\rho$.  Since $\nu_2$ was a geodesic, so is the path with label $v_2^-v_2^+$.  
The concatenation of the subpath labeled by $v_1^+w_2$, and the subpath labeled by $v_2^-v_2^+$ forms a geodesic triangle in $\Ga(\Z^n, Z_\rho)$.  Replacing the subpath of $\alpha'$ labeled by $v^+_1w_2$ with the geodesic labeled by $v_2^-v_2^+$  yields a new path contained in the $\delta'$--neighborhood of $\alpha'$.  
This new path has label 
\[
v^-_1u'_1v^-_2v^+_2u_2\ldots u_{l}w_{l+1}
\]
and is contained in the $(2\delta'+\ell(\mu_1))$--neighborhood of $\alpha$.

As above, the properties of $Q$ allow us to move all the edges of the subpath labeled by $v^+_2$ past all  the edges from $u_{2}$ to obtain a path with label
 $$
 v^-_1u'_1v^-_2u'_2v^+_2 w_3\ldots u_{l}w_{l+1}
 $$ 
 which is contained in the $(2\delta' + \ell(\mu_1)+\ell(\mu_2))$--neighborhood of $\alpha$. 

Continuing this process, we eventually obtain a path $\alpha''$ with label 
$$
v^-_1 u'_1v^-_2u'_2\ldots u'_l v^-_{l+1}v^+_{l+1}
$$
  where every letter in $v^+_{l+1}$ has non-negative image under $\rho$, every letter of $v_i^-$ has negative image under $\rho$ for $1\leq i\leq l+1$, and every letter in each $u'_j$ is in $Q$.  This path $\alpha''$  is contained in the  $((k+1)\delta' + k)$--neighborhood of $\alpha$ since $l\leq k$ and the total number of edges from $Q$ is $k$.   Let $\mu_i'$ be the subpaths of $\alpha''$ with labels $u_i'$.

We now move each edge of the subpath labeled by $v^-_{l+1}$ past all the edges in the subpath labeled by $u'_l$ by using the properties of $Q$ in an analogous way as above. This yields a path with label  
$$
v^-_1 u'_1v^-_2u'_2\ldots v^-_lv^-_{l+1} u''_lv^+_{l+1}
$$
 which is contained in the $((k+1)\delta' +k + \ell(\mu_l'))$--neighborhood of $\alpha$. Again using the properties of $Q$, we move each edge of the subpath labeled by $v^-_lv^-_{l+1}$ past each  edge of $\mu'_{l-1}$.  By continuing this process, we eventually obtain a path $\tau'$ with label 
  $$
  v^-_1 v^-_2 \ldots v^-_{l+1}(u''_1u''_2 \ldots u''_l)v^+_{l+1}
  $$
  that is contained in the neighborhood of $\alpha$ of radius \[(k+1)\delta'+k+\ell(\mu_1')+\cdots+\ell(\mu_l')=(k+1)\delta'+2k.\]
 For the final step, we form the path $\tau$ by  replacing the subpath $\tau''$ of $\tau'$ labeled by $v_1^-v_2^-\ldots v_{l+1}^-$ with a geodesic $\tau_1$ between its endpoints formed in the following way.  
 
Let $v=v_1^-v_2^-\ldots v_{l+1}^-$, and note that $\rho(v)=\displaystyle \sum_{i=1}^{l+1}\rho(v_i^-)<0$.  Let $m=\left \lfloor \dfrac{-\rho(v)}{C_\rho} \right \rfloor$.  Since $\rho$ is a homomorphism and no element of $Z_\rho$ has image under $\rho$ whose absolute value is larger than $C_\rho$, we have $m\leq \|v\|_{Z_\rho}\leq m+1$, and $\|v\|_{Z_\rho}=m$ exactly when $\dfrac{-\rho(v)}{C_\rho}\in \Z$.  If $\|v\|_{Z_{\rho}}=m$ we must have $v=v_1'v_2'\ldots v_m'$ where $\rho(v_i')=-C_\rho$ for all $i$.    In this situation we let $\tau_1$ be the geodesic between the endpoints of $\tau''$   labeled by $v_1'v_2'\ldots v_m'$.  Now suppose $\|v\|_{Z_\rho}=m+1$, and recall that by definition there is an element $z\in Z_\rho$ with $\rho(z)=-C_\rho$.  Consider the element $z^{-m}v\in \Z^n$.  Then $-C_\rho< \rho(z^{-m}v)<0$, and so $z^{-m}v\in Z_\rho$.  In this case, we let $\tau_1$ be the geodesic between the endpoints of $\tau''$ labeled by $z^m(z^{-m}v)\in \Z^n$.  In either case, the geodesic $\tau_1$ has the property that each edge is labeled by an element whose image under $\rho$ is negative.  Moreover,  since $\tau''$ is the concatenation of $l+1\leq k+1$ geodesics, $\tau_1$ is contained in its $(k+1)\delta'$--neighborhood.  

We claim $\tau$ satisfies the first statement of the lemma.  Let $\tau_2$ be the subpath of $\tau$ with label $u_1''u_2'' \ldots u_l''$ and $\tau_3$ the subpath with label $v^+_{l+1}$.  Then the following hold by construction: $\tau_1$ and $\tau_3$ are geodesics;  each letter in $\Lab(\tau_2)$ is in $Q$; every letter in $\Lab(\tau_1)$ has negative image under $\rho$;  and each letter of $\Lab(\tau_3)$ has non-negative image under $\rho$.   Further, $\tau$ is contained in the $(2(k+1)\delta' +2k)$--neighborhood of $\alpha$ and has the same endpoints as $\alpha$.  Finally, notice that this process does not increase the length of the path we started with, and so the final bullet point of the  statement of the lemma  holds.  This completes the proof of the first statement of the lemma. 

The final bullet point immediately implies that if $\alpha$ is a geodesic, then so is $\tau$. To  prove the second part of the ``moreover" statement, notice that the only times in this procedure when we do not get a bound on the Hausdorff distance between paths at successive stages is when we have a subpath whose label is of the form $vw$, where $v,w\in \Z^n$, which we replace with a geodesic $\nu$ between its endpoints. (We think of the final step in the above procedure as iterating this $l$ times.) In general, the subpath labeled by $vw$ is only a concatenation of geodesics and may not be a geodesic itself.  In particular, there may be backtracking at the concatenation point, and so we do not always get a bound on the Hausdorff distance at this step.  However, if there is backtracking, then $\ell(\nu)$ is strictly less than the length of the subpath labeled by $vw$.  In particular, we must have $\ell(\tau)<\ell(\alpha)$.  Since $\alpha$ and $\tau$ have the same endpoints, this contradicts our assumption that $\alpha$ is a geodesic.  Therefore, whenever $\alpha$ is a geodesic, the subpath with label $vw$ must be a geodesic, and thus replacing this subpath with $\nu$ results in a  bound on the Hausdorff distance between the paths. (In fact, in this situation we may skip this step altogether.) Therefore, in this case we conclude that the Hausdorff distance between $\alpha$ and $\tau$ is at most $2(k+1)\delta' +2k$.  Finally, since $\tau$ is a geodesic, the subpath $\tau_2$ of $\tau$  is also a geodesic and has length $\displaystyle \sum_{i=1}^l\ell(\mu_i'')= \sum_{i=1}^l\ell(\mu_i)=k$.  This label is an element of $H$, and so by Lemma \ref{control}, we conclude that $k\leq k_0$.  Therefore, the Hausdorff distance between $\alpha$ and $\tau$ is uniformly bounded by $2(k_0+1)\delta' +2k_0$.  This concludes the proof of the lemma.
\end{proof}

\begin{lem}\label{moveQ} Suppose $\mu=\mu_1\mu_2$  is a path in $\Ga(G,Q\cup Z_\rho)$ with  $\Lab(\mu_1)= x_1x_2 \ldots x_m$ and  $Lab(\mu_2) = q_1q_2 \ldots q_k$, where $ x_j \in Z_\rho$ and $q_i \in Q$. Further assume that $\rho(x_j) \geq 0 $ for all $1 \leq j \leq m$.  Let $\nu$ be the path with the same endpoints as $\mu$ provided by Lemma \ref{newpaths}, so that $\nu=\nu_1\nu_2$ satisfies $Lab(\nu_1) = q_1' q_2' \ldots q_k'$ and $Lab(\nu_2) = x_1x_2 \ldots x_m=\Lab(\mu_1)$, where $q_i' \in Q$.

If $v$ is any vertex on $\mu$ (respectively, $\nu$), then there exists a vertex $v'$ on $\nu_2$ (respectively, $\mu_1$) such that $d(v,v')\leq k$.  In particular, if $\mu$ is a geodesic, then we have $d(v,v')\leq k_0$, where $k_0$ is the constant from Lemma \ref{control}.
\end{lem}

\begin{proof} We use the procedure described in Lemma \ref{newpaths} in the special case $\mu=\mu_1\mu_2$ to obtain the new path $\nu=\nu_1\nu_2$.  In this situation, we need only apply the first step of the procedure, which consists of moving $\Lab(\mu_1)$ past each letter $q_i$ in $\Lab(\mu_2)$. Each time we move $\Lab(\mu_1)$ past some $q_i$, we form a new path at Hausdorff distance one from the previous path; see Figure \ref{squares}. Indeed, moving $x_m$ past $q_1$ yields a new path with the label $x_1x_2 \ldots x_{m-1} q_{1,1}x_mq_2q_3\ldots q_k$ at Hausdorff distance one from $\mu$.  After moving each letter $x_j$ of $\Lab(\mu_1)$ past $q_1$ we have a path with label $q_{1,m}x_1\ldots x_m q_2\ldots q_k$, which is still at Hausdorff distance one from $\mu$. Repeating this procedure $k$ times to move $\Lab(\mu_1)$ past each $q_i$ forms a rectangle. Setting $q_i'=q_{i,m}$ in the statement of the lemma results in a path of the desired form. 

\begin{figure}
\centering
\def\svgscale{0.6}
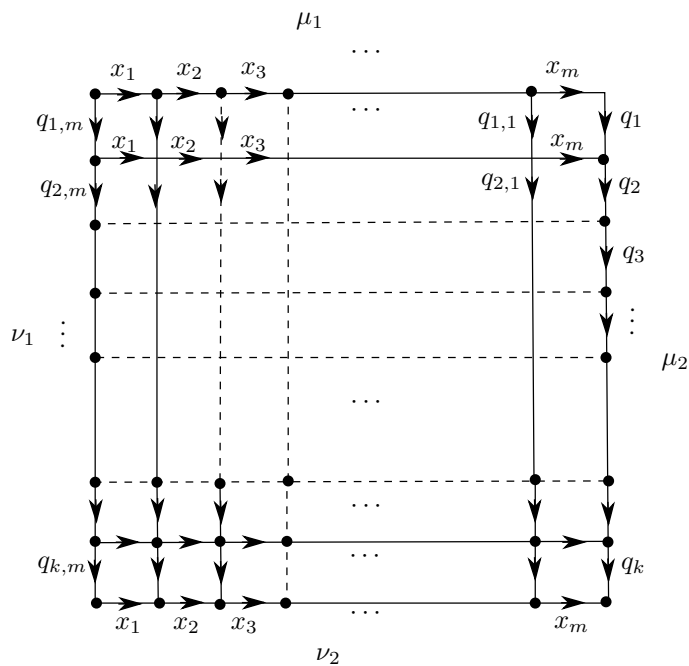
\caption{Applying Lemma \ref{newpaths} to the path $\mu=\mu_1\mu_2$.}
\label{squares}
\end{figure}

It can be seen from Figure \ref{squares} that any vertex $v$ of $\mu$ (respectively, $\nu$) is at distance at most $k$ from a vertex of the path $\nu_2$ (respectively, $\mu_1$). Indeed, if $v$ is a vertex of $\mu$ (respectively, $\nu$), then there is a vertical path consisting of at most $k$ edges from $v$ to a vertex of $\nu_2$ (respectively, $\mu_1$).
The final statement of the lemma follows from applying Lemma \ref{control}.
\end{proof}

The proof of the following lemma is similar to that of \cite[Proposition 4.6]{Amen}.  However, since in our situation $\Ga(\Z^n,Z_\rho)$ is only quasi-isometric to a line and not actually a line, some additional subtleties arise.

\begin{lem} \label{hyp}
$\Ga(G, Q\cup Z_\rho)$ is hyperbolic. 
\end{lem}

\begin{proof}
We will consider paths in the Cayley graph $\Ga(G,Q\cup Z_\rho)$.   We consider $\Gamma(\Z^n,Z_\rho)$ as a subgraph of $\Ga(G,Q\cup Z_\rho)$ in the natural way. Recall that $\delta'\geq 0$ is a hyperbolicity constant of $\Gamma(\Z^n,Z_\rho)$. Let $k_0$ be the constant from Lemma \ref{control}, and let $A=2(k_0+1)\delta'+2k_0$ be the constant from Lemma \ref{newpaths}.

Recall that a \emph{geodesic bigon} is a concatenation $\alpha_1\alpha_2$ where $\alpha_1$ and $\alpha_2$ are geodesics such that the initial point of $\alpha_1$ is the endpoint of $\alpha_2$ and the endpoint of $\alpha_1$ is the initial point of $\alpha_2$.  We will show that geodesic bigons in $\Ga(G,Q\cup Z_\rho)$ are $\delta$--slim for  $\delta=4\delta'+3k_0+2A$; that is, each side is contained in the  $\delta$--neighborhood of the other. Papasoglu shows that this  suffices to prove hyperbolicity of $\Ga(G,Q\cup Z_\rho)$ in  \cite[Theorem~1.4]{Papa}. (While Papasoglu stated this result only for Cayley graphs of groups with respect to finite generating sets, it is straightforward to see that the proof only relies on the fact that the space is a connected graph with the simplicial metric.  This was explicitly noted, for example, in \cite{NeumannShapiro}.)    

Consider a geodesic bigon $\alpha_1'\alpha_2'$ in $\Gamma(G,Q\cup Z_\rho)$.  By Lemma \ref{newpaths}, there is a pair of $A$--Hausdorff close geodesics $\alpha_1$ and $\alpha_2$ with the same endpoints  such that for $i=1,2$ the label of each $\alpha_i$ has the form $p_iu_iq_i$,  where $u_i$ is a  word in $Q$ with length $\leq k_0$, and $p_i,q_i\in \Z^n$ are words $p_i=p_i^1\ldots p_i^{j_i}$ and $q_i=q_i^1\ldots q_i^{m_i}$ satisfying $\rho(p_i^s)\leq 0$ and  $\rho(q_i^r) \geq 0$ for all $1\leq s\leq j_i$ and $1\leq r\leq m_i$. Thus we have  $$\operatorname{Lab}(\alpha_1\alpha_2)= (p_1u_1q_1)(p_2u_2q_2).$$   We will show that the bigon $\alpha_1\alpha_2$ is $(4\delta'+3k_0)$--slim, which will prove the result. 

Let $\alpha_i=\sigma_i\mu_i\xi_i$, where $\operatorname{Lab}(\sigma_i)=p_i$, $\operatorname{Lab}(\mu_i)=u_i$, and $\operatorname{Lab}(\xi_i)=q_i$ for $i=1,2$, so that 
\[
\alpha_1\alpha_2=(\sigma_1\mu_1\xi_1)(\sigma_2\mu_2\xi_2).
\]
 
Suppose that $y_0$ is a point on $\alpha_1$. We will find a point on $\alpha_2$ at distance at most $4\delta'+3k_0$ from $y_0$. (If $y_0$ is a point on $\alpha_2$, an analogous argument will hold.)  If $y_0$ lies on $\mu_1$, then since $\|u_1\|_{Q\cup Z_\rho}\leq k_0$, there are points on  $\sigma_1$ and $\xi_1$ at distance at most $k_0$ from $y_0$. Moreover, if $\sigma_1$ (respectively, $\xi_1$) has length at most $2\delta'+k_0$ and $y_0$ lies on $\sigma_1\mu_1$ (respectively, on $\mu_1\xi_1$) then $y_0$ is at distance at most $2\delta'+2k_0$ from $\alpha_2$. Thus it suffices to assume that $y_0$ lies on $\sigma_1$ or $\xi_1$, at least $2\delta'+k_0$ from the common endpoint with $\mu_1$, and find a point on $\alpha_2$ at distance at most $2\delta'+k_0$ from $y_0$.  We will assume $y_0$ lies on $\sigma_1$.  If $y_0$ lies on $\xi_1$, then a symmetric argument will prove the result.

First, we use the fact that $\Z^n$ is abelian to replace the concatenation of geodesics $\xi_1\sigma_2$ with a geodesic $\nu^-\nu^+$ in $\Gamma(\Z^n,Z_\rho)$, where the label of each edge in $\nu^+$ has non-negative image under $\rho$ and the label of each edge in $\nu^-$ has negative image under $\rho$.  (In the case $y_0$ lies on $\xi_1$, then the first step is to replace $\xi_2\sigma_1$ with the geodesic $\nu^-\nu^+$, instead.)  Thus we have a path $(\sigma_1\mu_1\nu^-)(\nu^+\mu_2\xi_2)$ whose label is equal  to $\Lab(\alpha_1\alpha_2)$ when both are considered as elements of the group $G$; see Figure \ref{fig:bigon}.
By Lemma \ref{moveQ}, there are paths $\omega^-\mu_1'$ and $\mu_2'\omega^+$ at Hausdorff distance at most $k_0$ from $\mu_1\nu^-$ and $\nu^+\mu_2$, respectively, where $\omega^+,\omega^-$ are geodesics in $\Z^n$ with the same labels as $\nu^+, \nu^-$, respectively, and $\mu_1',\mu_2'$ are geodesic with labels $u_1,u_2\in H$, respectively. Consequently, $\mu_i'$ has length $\leq k_0$ for $i = 1,2$.   

\begin{figure}
\centering
\def\svgscale{0.5}
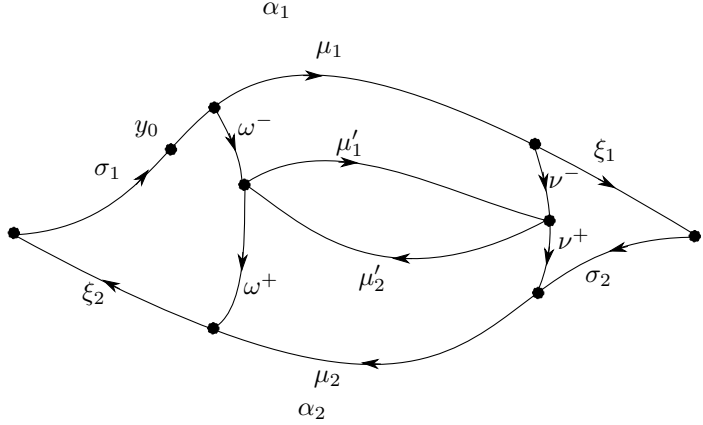
\caption{The decomposition of the bigon $\alpha_1\alpha_2$ in the case $y_0$ lies on $\sigma_1$.}
\label{fig:bigon}
\end{figure}

Since $\alpha_1\alpha_2$ is a loop, its label is the identity element of $G$, and so the image of $\Lab(\alpha_1\alpha_2)$ under the natural projection $G\to \Z^n$ is  the identity element 0 of $\Z^n$.  In particular,  we have $$p_1 + q_1 + p_2 + q_2 =\Lab(\sigma_1) + \Lab(\xi_1) + \Lab(\sigma_2) + \Lab(\xi_2) = 0,$$
where $p_i,q_i$ are considered as elements of $\Z^n$.
 Since $q_1 + p_2 $ and $ \Lab(\nu^-)+\Lab(\nu^+)$ represent the same element of $\Z^n$, we have $$p_1 + \Lab(\nu^-) + \Lab(\nu^+) + q_2 =0.$$ Since $\omega^-, \omega^+$ have the same labels as $\nu^-, \nu^+$ respectively, this gives us that $$p_1 + \Lab(\omega^-) + \Lab(\omega^+) + q_2 =0.$$ Therefore the path $\omega^+\xi_2\sigma_1\omega^-$  is a loop in $\Gamma(\Z^n,Z_\rho)$ and hence in $\Gamma(G,Q\cup Z_\rho)$.  Moreover,  since $\nu^-\nu^+$ is a geodesic, so is $\omega^- \omega^+$. Consequently, $\sigma_1 \omega^- \omega^+ \xi_2$ is a geodesic triangle  in $\Gamma(\Z^n,Z_\rho)$, and $\mu'_1 \mu'_2$ is a geodesic bigon, as shown in Figure \ref{fig:bigon}.  

By the hyperbolicity of  $\Gamma(\Z^n,Z_\rho)$, there is a point $y_1$ on $\omega^- \omega^+\xi_2$ which is at distance at most $\delta'$ from $y_0$.  If $y_1$ lies on $\xi_2$, then we are done, as $\xi_2$ is a subpath of $\alpha_2$. 

Suppose $y_1$ lies on $\omega^-$.  Then Lemma \ref{moveQ} provides a point $y_2$ on $\nu^-$ at distance at most $k_0$ from $y_1$.  The path $\nu^-\nu^+\sigma_2^{-1}\xi_1^{-1}$ is a geodesic triangle in $\Gamma(\Z^n,Z_\rho)$ (recall that $\nu^-\nu^+$ is a geodesic), and therefore there is a point $y_3$ on $\xi_1\sigma_2$ which is at distance at most $\delta'$ from $y_2$.  Suppose that $y_3$ lies on $\xi_1$.  Then we  have $d(y_0,y_3)\leq \displaystyle \sum_{i=0}^2d(y_i,y_{i+1})\leq 2 \delta' +k_0$. However, $y_0$ and $y_3$ both lie on the geodesic $\alpha_1$.  Since $y_0$ is at least $2\delta' +k_0$ from the terminal endpoint of $\sigma_1$ while $y_3$ lies on $\xi_1$, this is a contradiction. We conclude that  $y_3$ must lie on $\sigma_2$.  As we still have $d(y_0,y_3)\leq  2\delta' + k_0$ and $\sigma_2$ is a subpath of $\alpha_2$,  we are done. 

Finally, suppose  that $y_1$ lies on $\omega^+$.  Then  Lemma \ref{moveQ} provides a point $y_2$ on $\nu^+$ at distance at most $k_0$ from $y_1$.  As in the previous paragraph, there must be a point $y_3$ on $\xi_1\sigma_2$ at distance at most $\delta'$ from $y_2$.  If $y_3$ lies on $\xi_1$, then we reach a contradiction exactly as in the previous paragraph.  Thus $y_3$ must lie on $\sigma_2$, and since $d(y_0,y_3)\leq 2\delta' + k_0 $, we are done.

Given a point $y_0$ on $\sigma_1$, we have found a point on $\alpha_2$ at distance at most $2\delta'+k_0$ from $y_0$.  Therefore any point $y_0'$ on $\alpha'_1$ is at distance at most $\delta=4\delta'+3k_0+2A$ from $\alpha'_2$. It follows that the bigon $\alpha'_1\alpha'_2$ is $\delta$--slim, and we conclude that $\Ga(G,Q\cup Z_\rho)$ is hyperbolic.  
\end{proof}


\begin{lem}[{cf. \cite[Proposition 4.6]{Amen}}] \label{focal}
If $Q$ is strictly confining, then the action $G \acts \Ga(G, Q\cup Z_\rho)$ is quasi-parabolic. Otherwise the action is lineal.
\end{lem}

\begin{proof} By Remark \ref{rem:Hnoloxos}, the subgroup $H$ cannot contain any loxodromic elements.  This implies that the action of $H$ on $\Ga(G,Q\cup Z_\rho)$ is either elliptic or parabolic. 

If $Q$ is not strictly confining, then for every $z\in \Z^n$ with $\rho(z) >0$, we have that $\gamma(z)(Q) = Q$, and consequently $Q = \gamma(z^{-1})(Q)$.  Let $h \in H$. Using the equivalent version of Definition \ref{def:confining}(b) given in Remark~\ref{rem:equivconditions}, there is a $z \in \Z^n$ such that $\rho(z) >0$ and $\gamma(z)(h) \in Q$. In particular, $h \in \gamma(z^{-1})(Q) =Q$. We conclude that $Q=H$. Thus $H$  has bounded orbits in the action of $G$ on $\Ga(G, Q\cup Z_\rho)$.  Since the action of $\Z^n$ on $\Ga(\Z^n, Z_\rho)$ is lineal, the action of $G$ on $\Ga(G,Q\cup Z_\rho)$ is  also lineal. 

On the other hand, if $Q$ is strictly confining, then there exists a $z\in \Z^n$  such that $\gamma(z)(Q)$ is a proper subset of $ Q$.   Remark \ref{rem:equivconditions} implies that we may choose such a $z$ satisfying  $\rho(z) >0$.  Note that we may extend the homomorphism $\rho$ to all of $G$ by setting $\rho$ to be identically zero on $H$. Since $\rho(z^n)$ grows linearly while the elements of $Q \cup Z_\rho$ have bounded image under the homomorphism $\rho$, the word  lengths $\|z^n\|_{Q\cup Z_\rho}$ must grow linearly as well. Thus $z$ acts loxodromically  on $\Gamma(G, Q \cup Z_\rho)$. 

Consider the strictly ascending chain 
$$
Q \subsetneq \gamma(z^{-1})(Q) \subsetneq \gamma(z^{-2})(Q)\subsetneq\cdots\subsetneq \gamma(z^{-k})(Q) \subsetneq\cdots.
$$ We will show that the word lengths of elements of $\gamma(z^{-i})(Q) \setminus \gamma(z^{-i+1})(Q)$ have linearly growing word length in $Q\cup Z_\rho$. This will then imply that $H$ has unbounded orbits in the action on $\Gamma(G,Q\cup Z_\rho)$.

Consider $h\in \gamma(z^{-i})(Q) \setminus \gamma(z^{-i+1})(Q)$ and suppose $\gamma(w)(h)\in Q$ for some $w\in \Z^n$. If $\rho(z^{i-1}w^{-1})\geq 0$, then \[\gamma(z^{i-1})(h)=\gamma(z^{i-1}w^{-1})\big(\gamma(w)(h)\big)\in Q.\] However, this contradicts our assumption on $h$, and so we must have $\rho(w)>\rho(z^{i-1})$.

By Lemma \ref{newpaths}, we may write $h$ as a geodesic word \[h=z_1\ldots z_r q_1\ldots q_s w_1\ldots w_t \] where $z_i\in Z_\rho$ satisfy $\rho(z_i)<0$, $q_i\in Q$,  $w_i\in Z_\rho$ satisfy $\rho(w_i)\geq 0$, and $r+s+t=\|h\|_{Q\cup Z_\rho}$. Moreover, we have $s\leq k_0$. Writing $v=z_1\ldots z_r$, $h'=q_1\ldots q_s$, and $w=w_1\ldots w_t$, we have $h=vh'w$, and since $h'\in H$ we must have $w=v^{-1}$. Thus, $h=vh'v^{-1}$ and $\gamma(v^{-1})(h)=h' \in Q^s\subseteq Q^{k_0}$. Recall that there exists $z_0$ such that $\gamma(z_0)(Q\cdot Q)\subseteq Q$. We thus have $\gamma(z_0^{k_0})(Q^{k_0})\subseteq Q$, and therefore $\gamma(z_0^{k_0}v^{-1})(h)\in Q$. By the previous paragraph, this implies that \[k_0\rho(z_0) +\rho(v^{-1})=\rho(z_0^{k_0}v^{-1})>\rho(z^{i-1}).\] 
Thus, $\rho(v^{-1})> (i-1)\rho(z)-k_0\rho(z_0)$. Since the image of any element of $Z_\rho$  under $\rho$ is bounded in absolute value by $C_\rho$, we have \[\|v\|_{Z_\rho} \geq \frac{(i-1)\rho(z)-k_0\rho(z_0)}{C_\rho} \ \ \ \text{ and } \ \ \  \|h\|_{Q\cup Z_\rho} \geq 2\left(\frac{(i-1)\rho(z)-k_0\rho(z_0)}{C_\rho}\right)+1.\]

Therefore $H$ has unbounded orbits in the action on $\Ga(G,Q\cup Z_\rho)$, and so the action of $H$ is parabolic.   Let $ \xi=\lim_{i\to\infty}z^{-i}$. We will show that $G$ fixes $\xi$.   Since $\rho(z^{-i}) < 0$, we have for any $q\in Q$ and each $i\geq 1$ that $$d_{Q \cup Z_\rho}(qz^{-i}, z^{-i}) = d_{Q \cup Z_\rho}(z^{-i}q', z^{-i}) = d_{Q \cup Z_\rho}(q', 1) =1$$ for some $q'\in Q$. Thus $Q$ fixes $\xi$. As $\Z^n$ also fixes $\xi$ (since the action of $\Z^n$ is lineal) and $Q\cup \Z^n$ generates $G$, it follows that all of $G$ fixes $\xi$.  Since $G$ has unbounded orbits and contains loxodromic elements, this shows that the action of $G$ on $\Ga(G,Q\cup Z_\rho)$ is either lineal or quasi-parabolic. Since $H$ acts parabolically, the action must be quasi-parabolic. 
\end{proof}

We now turn our attention to understanding the Busemann pseudocharacter associated to the action  of $G$ on $ \Ga(G, Q\cup Z_\rho)$.   We begin with a general fact about homomorphisms from $\Z^n$ to $\R$.  
\begin{lem} \label{lem:multiplekernel} For any homomorphisms $r, s\colon \Z^n \to \mathbb{R}$, the following are equivalent. 
\begin{enumerate} [(1)]
\item $r$ and $s$ are scalar multiples of each other. 
\item Either $r(z) \geq 0$ if and only if $s(z) \geq 0$ for all $z\in \Z^n$  or $r(z) \geq 0$ if and only if $s(z) \leq 0$ for all $z\in \Z^n$.
\end{enumerate} 
Moreover, if $r(z) \geq 0$ if and only if $s(z) \geq 0$ for all $z\in\Z^n$, then $r$ and $s$ are positive scalar multiples of each other, while if $r(z) \geq 0$ if and only if $s(z) \leq 0$ for all $z\in \Z^n$, then $r$ and $s$ are negative scalar multiples of each other.
\end{lem} 
\begin{proof}
 Clearly (1) implies (2). The homomorphisms $r$ and $s$ are given by $r(z)=v\cdot z$ and $s(z)=w\cdot z$ for some vectors $v,w\in \R^n$. The homomorphisms are proportional if and only if $v$ and $w$ are proportional, which is the case if and only if the orthogonal complements $v^\perp$ and $w^\perp$ in $\R^n$ are equal. If $v^\perp$ and $w^\perp$ are not equal, then they partition $\R^n\setminus (v^\perp \cup w^\perp)$ into four convex cones corresponding to the four possible pairs of signs of $v\cdot u$ and $w\cdot u$. Each of these regions contains an integer vector and therefore (2) also implies (1).
\end{proof}

\begin{lem} \label{lem:Busemann} Let $\beta$ be the Busemann pseudocharacter associated to the action of $G$ on $\Ga(G,Q\cup Z_\rho)$.  For any $g=hz\in G$, where $h\in H$ and $z\in \Z^n$, we have that $\beta(g)=\beta(z)$.  In other words,  $\beta$  is the composition of  the natural projection of $G$ to $\Z^n$ and the restriction of $\beta$ to $\Z^n$.  Moreover, the restriction of $\beta$ to $\Z^n$ is proportional to the homomorphism $\rho$. 
\end{lem}

\begin{proof}  
By Remark \ref{rem:Hnoloxos}, $H$ cannot contain any loxodromic isometries. Thus $\beta(h)=0$ for all $h\in H$.  Since $H$ is a normal subgroup of $G$, it follows from  \cite[Lemma~4.3]{Amen} that $\beta$ induces a homogeneous quasi-character on $G/H\cong \Z^n$.  In particular, for any $g=hz\in G$, we have $\beta(g)=\beta(z)$. To see this, let $D$ be the defect of $\beta$, and let $r$ be a positive integer. We have $(hz)^r=h'z^r$ for some $h'\in H$. Since $\beta(h')=0$ we obtain $|\beta((hz)^r)-\beta(z^r)|\leq D$. By using homogeneity, the left hand side is equal to $r|\beta(hz)-\beta(z)|$. Dividing both sides by $r$ and letting $r\to \infty$ we have $\beta(hz)=\beta(z)$. Additionally, since $\Z^n$ is abelian, all pseudocharacters are homomorphisms, and hence  $\beta$ is a homomorphism (see \cite[Proposition 2.65]{scl}).

We now turn to the ``moreover" statement of the lemma. For any $z'\in \Z^n$, to understand the action of $ \langle z' \rangle $ on $\Ga(G,Q\cup Z_\rho)$, we need only consider the action of $\Z^n$ on $\Ga(\Z^n,Z_\rho)$. This is because $\Ga(\Z^n, Z_\rho)$ is a quasi-line and hence quasi-convex in $\Ga(G, Q \cup Z_\rho)$. Now $\beta(z') > 0$ if and only if $z'$ is a loxodromic element in the action on $\Ga(G, Q\cup Z_\rho)$ with repelling fixed point $\xi$.  This occurs if and only if $z'$ is loxodromic with respect to the action of $\Z^n$ on $\Ga(\Z^n,Z_\rho)$ with repelling fixed point $\xi$.  By \cite[Lemma 4.15]{ABO}, this is true if and only if  $\rho(z') >0$. 

Recall that $\Ga(\Z^n,Z_\rho)$ is a quasi-line. In an action on a quasi-line, all elements are either elliptic or loxodromic.  Thus if $\beta(z') =0$, then $z'$ is an elliptic element in the action of $\Z^n$ on $\Ga(\Z^n,Z_\rho)$ and hence in the action of $G$ on $\Ga(G, Q\cup Z_\rho)$.  It again follows from \cite[Lemma 4.15]{ABO} that this happens if and only if $\rho(z') = 0$. 

We have shown that $\beta(z') \geq 0$ if and only if $\rho(z') \geq 0$. By Lemma \ref{lem:multiplekernel}, we see that  $\rho$ and $ \beta$ are positive scalar multiples of each other.  
\end{proof}

  Lemmas \ref{hyp}, \ref{focal}, and \ref{lem:Busemann} prove Theorem \ref{thm:main}(i), (ii), and (iii), respectively.

\subsection{Confining subsets from actions}
Throughout this section, we fix a group $G=H\rtimes_\gamma \Z^n$ and a cobounded action $G\curvearrowright X$ on a hyperbolic space with a global fixed point $\xi \in \partial X$.   We additionally assume that the associated Busemann pseudocharacter $\beta$ satisfies $\beta(H)=0$. Thus $\beta$ restricts to a pseudocharacter $\Z^n\to \R$. As before, we see that $\beta$  is a homomorphism and, in fact, $\beta(hz)=\beta(z)$ for $h\in H,z\in \Z^n$.

We also fix the following data for the rest of the subsection.  Let $\delta$ be a constant of hyperbolicity for $X$.  
Since $\beta$ is non-zero and $\beta(H)=0$, there is an element $z_0\in \Z^n$ such that $\beta(z_0)\neq 0$.  Thus $z_0$ is a loxodromic element, which must fix the point $\xi$. Let $\nu\neq \xi$ be the other fixed point of $z_0$ in $\partial X$.  Let $c\colon (-\infty,\infty)\to X$ be a $(1,20\delta)$--quasi-geodesic from $\nu$ to $\xi$.  
Such a $(1,20\delta)$--quasi-geodesic $c$ always exists between any two points of $\partial X$ (see \cite[Remark~2.16]{boundaries}, for instance). We fix a basepoint $x=c(0)$ of $X$.

Our first goal is to prove the following proposition.   
This proposition has been used implicitly (in the $n=1$ case) in all three of the papers \cite{Bal, AR, Amen},  but to the best of the authors' knowledge it has never received a detailed proof. Because of its fundamental importance, we include a proof here.

\begin{prop}
\label{prop:smalltranslation}
Let $G=H\rtimes_\gamma \Z^n$ be a group acting on a hyperbolic space $X$ and fixing a point $\xi\in \partial X$, and let $\beta\colon G\to \R$ be the associated Busemann pseudocharacter.  Assume $\beta(H)=0$, and fix a basepoint $x\in X$ as above.  There exists a function $A\colon \R_{\geq 0}\to \R$ and a constant $B>0$ such that the following holds.  For any $g\in G$ and $z\in \Z^n$ with $d(x,gx)\leq \constfive$ and $\beta(z)\leq -A(\constfive)$, we have $d(gz(x),z(x))\leq |\beta(g)|+B$.
\end{prop}

For  simplicity of notation, we use $d$ to denote the metric on $X$ throughout the proof of the proposition.  We may define the Busemann pseudocharacter $\beta$ associated to the action $G\curvearrowright X$ in the following way. First of all, we define $q\colon G\to \R$ by \[q(g)=\limsup_{t\to\infty} \Big(d(gx,c(t))-d(x,c(t))\Big)=\limsup_{t\to\infty} \Big(d(gc(0),c(t))-d(c(0),c(t))\Big).\] The Busemann pseudocharacter is then the homogenization $\beta$ defined by \[\beta(g)=\lim_{n\to\infty} \frac{q(g^n)}{n}.\]

Let $r_0$ be a constant such that any two $(1,20\delta)$--quasi-geodesic rays in $X$ with the same endpoint on $\partial X$ are eventually $r_0$--Hausdorff close and any two bi-infinite $(1,20\delta)$--quasi-geodesics with the same endpoints are $r_0$--Hausdorff close.  For any $g\in G$, the ray $c|_{[0,\infty)}$ and its translate $gc|_{[0,\infty)}$ are both $(1,20\delta)$--quasi-geodesic rays that share the endpoint $\xi$ and thus  are eventually $r_0$--Hausdorff close. Specifically, there are numbers $t_0=t_0(g)$ and $s_0=s_0(g)\geq 0$ depending on $g$ and $x=c(0)$ so that $c|_{[t_0,\infty)}$ and $gc|_{[s_0,\infty)}$ are $r_0$--Hausdorff close and $d(c(t_0),gc(s_0))\leq r_0$.   In other words $s_0$ is roughly how long it takes for the ray $gc|_{[0,\infty)}$ to become close to the ray $c|_{[0,\infty)}$. This depends only on $d(x,gx)$, and $s_0(g)$ may be chosen smaller than a function of $d(x,gx)$.  We consider the difference $l=t_0-s_0$ as the amount that $g$ ``shifts'' the quasi-geodesic $c$, which may be positive or negative.

We will prove the proposition in a series of lemmas.  The first says that $g$ uniformly shifts the entire quasi-geodesic by the same amount $l$. 

\begin{lem}
\label{lem:shift}
There exists a constant $\constone$ such that  $d(c(s+l),gc(s))\leq \constone$ for any $g\in G$ and for all $s\geq s_0$. Additionally, for each $n\in \Z_{\geq 0}$ we have \[d\big(g^nc(s),c(s+nl)\big)\leq n\constone\] for all $s\geq \max\{s_0, s_0+ (n-1)l\}$. 
\end{lem}

\begin{proof}
For each $s\geq s_0$ we have that $gc(s)$ is $r_0$--close to some point $c(t)$ with $t\geq t_0$. We have \[(t-t_0)-20\delta \leq d\big(c(t_0),c(t)\big)\leq d\big(gc(s_0),gc(s)\big)+2r_0 \leq (s-s_0)+20\delta+2r_0.\] Thus, $(t-t_0)\leq (s-s_0)+40\delta+2r_0$. By the same reasoning, we find $(s-s_0)\leq (t-t_0)+40\delta+2r_0$. In other words, \[|(s-s_0)-(t-t_0)|\leq 40\delta+2r_0.\] 
We may rewrite this as \begin{equation}\label{eqn:stl}|(s-t)+l|\leq 40\delta+2r_0.\end{equation} We conclude that  \[d\big(gc(s),c(s+l)\big)\leq d\big(gc(s),c(t)\big)+d\big(c(t),c(s+l)\big)\leq r_0 + |(s+l)-t|+20\delta\leq 60\delta+3r_0,\] where the first inequality follows from the triangle inequality, the second from our choice of $t$ and the fact that $c$ is a $(1,20\delta)$--quasi-geodesic, and the third from \eqref{eqn:stl}. Setting $\constone=60\delta+3r_0$ gives the first inequality in the statement of the lemma.

For the second inequality, note that $d(g^2c(s),gc(s+l))\leq \constone$ for $s\geq s_0$. By the first inequality, the point $gc(s+l)$ is in turn $\constone$--close to the point $c(s+2l)$ as long as $s+l$ is also at least $s_0$. In other words, as long as $s$ and $s+l$ are both at least $s_0$, the point $g^2c(s)$ is $2\constone$-close to $c(s+2l)$. Thus, $g^2$ shifts points of $c$ by $2l$, but the constant of closeness degrades from $\constone$ to $2\constone$ and $s_0$ degrades to $\max\{s_0,s_0+l\}$.
An induction argument using this reasoning gives the second inequality, completing the proof of the lemma.
\end{proof}

The next lemma shows that the shift constant $l$ and the closeness constant $\constone$ give bounds for $\beta(g)$.

\begin{lem}
\label{lem:quasimorphbounds}
We have 
\[-l-\constone\leq \beta(g)\leq -l+\constone.\]
\end{lem}

\begin{proof}
 By the triangle inequality, we have that $d\big(g^{-n}c(t),c(0)\big)-d\big(c(t),c(0)\big)$ is bounded above by 
\begin{equation}\label{eqn:upperbd}
    d\big(g^{-n}c(t),c(t-nl)\big)+d\big(c(t-nl),c(0)\big)-d\big(c(t),c(0)\big)
\end{equation}
and bounded below by 
\begin{equation}\label{eqn:lowerbd}
    d\big(c(t-nl),c(0)\big) - d\big(g^{-n}c(t),c(t-nl)\big) - d\big(c(t),c(0)\big).
\end{equation}
If $t\geq \max\{s_0,s_0+(n-1)l\}+nl$, then applying Lemma \ref{lem:shift} with $s=t-n\ell$ and the fact that $c$ is a $(1,20\delta)$--quasi-geodesic implies that \eqref{eqn:upperbd} is at most  $n\constone+(t-nl+20\delta)-(t-20\delta)=n\constone-nl+40\delta$ and \eqref{eqn:lowerbd} is at least $ -nl - nD -40\delta$.  Thus, \[-nl-n\constone-40\delta \leq d\big(g^{-n}c(t),c(0)\big)-d\big(c(t),c(0)\big)\leq -nl+n\constone+40\delta\] for all $t$ sufficiently large. Taking the $\limsup$ on all sides and using that $d(g^{-n}c(t),c(0))=d(c(t),g^nc(0))$, we obtain 
\[-nl-n\constone-40\delta \leq q(g^n)\leq -nl+n\constone+40\delta.\]
Dividing these inequalities by $n$ and letting $n$ go to infinity gives the bounds on $\beta$.
\end{proof}

Combining Lemma \ref{lem:shift} with Lemma \ref{lem:quasimorphbounds} we obtain the following corollary.

\begin{cor}
\label{cor:shiftfunction}
There is a constant $\consttwo>0$ so that for any $g\in G$, if $s\geq s_0(g)$ then $d\big(c(s-\beta(g)),gc(s)\big)\leq \consttwo$.
\end{cor}

\begin{proof}
By the triangle inequality, we have \[d\big(c(s-\beta(g)),gc(s)\big)\leq d\big(c(s-\beta(g)),c(s+l)\big)+d\big(c(s+l),gc(s)\big).\] By Lemma \ref{lem:quasimorphbounds}, $|l+\beta(g)|\leq \constone$. Thus, the first quantity on the right hand side is bounded by $\constone+20\delta$. As long as $s\geq s_0$, the second quantity on the right hand side is bounded by $\constone$. Taking $\consttwo=2\constone+20\delta$ completes the proof. 
\end{proof}

To prove Proposition \ref{prop:smalltranslation} we need one more result.

\begin{lem}
\label{lem:abelianaxisdiam}
If $z\in \Z^n$ then $d\big(zx,c(-\beta(z))\big)\leq \consttwo$.
\end{lem}

\begin{proof}
Since $x=c(0)$, the orbit of $x$ under $\Z^n$ lies in the orbit of $c$ under $\Z^n$. Recall that $z_0\in \Z^n$ is a loxodromic isometry whose fixed points are the endpoints $\xi$ and $\nu$ of $c$.  Since $\Z^n$ is abelian, every element of $\Z^n$ fixes these endpoints. Thus, for any $z\in \Z^n$, $c$ and $zc$ are $r_0$--Hausdorff close. 

Corollary \ref{cor:shiftfunction} implies that $zc(s)$ is $\consttwo$--close to $c(s-\beta(z))$ for all $s$ sufficiently large. In fact, since $c$ and $zc$ are $r_0$--Hausdorff close, we have that $zc(s)$ is $\consttwo$--close to $c(s-\beta(z))$ for all $s\geq 0$, by chasing through the proofs of the above results. In particular, $zx=zc(0)$ is $\consttwo$--close to $c(-\beta(z))$.
\end{proof}

We are now ready to prove Proposition \ref{prop:smalltranslation}.

\begin{proof}[Proof of Proposition \ref{prop:smalltranslation}]
Fix $\constfive\in \R_{\geq 0}$, and let $g\in G$ satisfy $d(x,gx)\leq \constfive$.  Let 
$s_0=s_0(g)$ be the constant defined after the statement of Proposition \ref{prop:smalltranslation} for this element $g$. Recall that $s_0$ is bounded above in terms of $N$.
By Corollary \ref{cor:shiftfunction}, if $s\geq s_0$ then 
\begin{equation}\label{eqn:s-pg}
d\big(c(s-\beta(g)),gc(s)\big)\leq \consttwo.
\end{equation}
 By Lemma \ref{lem:abelianaxisdiam}, if $z\in\Z^n$ then 
 \begin{equation}\label{eqn:zx}
 d\big(zx,c(-\beta(z))\big)\leq \consttwo.
 \end{equation}

Now suppose that $z\in\Z^n$ satisfies $\beta(z)\leq -s_0+\beta(g)$. Applying the triangle inequality three times we find that $d(zx,gzx)$ is at most \begin{equation}\label{eqn:threetriangles}d\big(zx,c(-\beta(z))\big)+d\big(c(-\beta(z)),gc(-\beta(z)+\beta(g))\big)+d\big(gc(-\beta(z)+\beta(g)),gc(-\beta(z))\big)+d\big(gc(-\beta(z)),gzx\big).\end{equation} Equation \eqref{eqn:zx} bounds the first and last summands in \eqref{eqn:threetriangles} by $E$.  Since $-\beta(z)+\beta(g)\geq s_0$,  applying \eqref{eqn:s-pg} bounds the second summand in \eqref{eqn:threetriangles} by $E$.  The third summand in \eqref{eqn:threetriangles} is bounded by $|\beta(g)|+20\delta$ since $c$ is a $(1,20\delta)$--quasi-geodesic.  

We define \[
A(N) = \sup \{s_0(g)-\beta(g) \mid d(x,gx)\leq \constfive\}
\]  
and
 \[
B=3\consttwo+20\delta.
\]
Since $s_0(g)$ depends only on $d(x,gx)$ and $|\beta(g)|\leq d(x,gx)$, the function $A(N)$ is well-defined.  As $\consttwo$ is uniform, this completes the proof.
\end{proof}

We are now ready to prove the main result of this section, which will immediately imply Theorem~\ref{prop:main}.

\begin{thm} \label{thm:actiontoconf}
Let $G = H \rtimes_\gamma \Z^n$, and let $G\curvearrowright X$ be a cobounded lineal or quasi-parabolic action on a hyperbolic space $X$.  Let $\beta$ be the Busemann pseudocharacter associated to this action, and assume that $\beta(H) =0$, so that $\beta$ restricts to a homomorphism $\Z^n\to \R$. There exists a subset $Q\subseteq H$ which is confining under   $\gamma$ with respect to  $\beta$ such that $X$ is $G$--equivariantly quasi-isometric to $\Gamma(G,Q\cup Z_\beta)$, where $Z_\beta$ is as in \eqref{eqn:Zrho}.  Moreover, if $G\curvearrowright X$ is lineal, then $Q$ is not strictly confining (and therefore $Q=H$), while if $G\curvearrowright X$ is quasi-parabolic, then $Q$ is strictly confining. 
\end{thm}

\begin{proof}
By the Schwarz--Milnor Lemma (see Lemma \ref{lem:MS}), we may assume without loss of generality that $X$ is a Cayley graph $\Gamma(G,Y)$. 
As described at the beginning of this section, we let $\xi$ be the fixed point of $G$ in $\partial \Gamma(G,Y)$ and $\nu$  be the other fixed point of $\Z^n$. Let $\delta$ be the constant of hyperbolicity of $\Gamma(G,Y)$. We  choose $c$ to be a $(1,20\delta)$--quasi-geodesic with $c(\infty)=\xi$ and  $c(-\infty)=\nu$. We let $x=c(0)$ and choose $r_0$ to be a constant such that any two $(1,20\delta)$--quasi-geodesics rays with the same endpoint are eventually $r_0$--Hausdorff close. 

A slight difficulty presents itself since the quasi-geodesic $c$ may not pass through the identity 1 of $G$, which is the natural basepoint of $\Gamma(G,Y)$. To fix this, we note the following slight modification of Proposition \ref{prop:smalltranslation}.
\begin{claim}
\label{claim:smalltransimprovement}
There is a function $A_0\colon \R_{\geq 0}\to \R$ and a constant $B_0>0$ (depending on $x$) such that the following holds.    For any $g\in G$ with $d_Y(1,g)\leq \constfive$ and $z\in \Z^n$ with $\beta(z)\leq -A_0(\constfive)$, we have $d_Y(gz,z)\leq |\beta(g)|+B_0$. 
\end{claim}

\begin{proof}[Proof of Claim]
Let $B$ be the constant from Proposition \ref{prop:smalltranslation}, and let $z\in \Z^n$.  Two applications of the triangle inequality yield
\[
d_Y(gz,z)\leq d_Y(gzx,zx)+2d_Y(1,x).
\]
Note that $d_Y(x,gx)\leq d_Y(1,g)+2d_Y(1,x)$. Thus if $d_Y(g,1)\leq \constfive$ we have $d_Y(gx,x)\leq \constfive+2d_Y(1,x)$. Additionally, if $z\in \Z^n$ with $\beta(z)\leq -A(\constfive+2d_Y(1,x))$ we have \[d_Y(gz,z)\leq d_Y(gzx,zx)+2d_Y(1,x)\leq |\beta(g)|+B+2d_Y(1,x).\] Taking $A_0\colon \R_{\geq 0}\to \R$ to be the function $A_0(\constfive)=A(\constfive+2d_Y(1,x))$ and $B_0$ to be the constant $B_0=B+2d_Y(1,x)$ proves the claim.
\end{proof}

We first define a subset of $H$; we will show that it is confining momentarily. Consider the ball $B_Y(1,B_0)$ of radius $B_0$ centered at the identity in $\Gamma(G,Y)$, that is, the set of elements in $G$ of word length at most $ B_0$ in the generating set $Y$. We moreover consider the intersection $B_Y(1,B_0)\cap H$.  Let $A_1=A_0(B_0)$, so that if $g\in B_Y(1,B_0)\cap H$ and $z\in\Z^n$ with $\beta(z)\leq -A_1$, then $d_Y(gz,z)\leq B_0$ (since $\beta(g) =0$). We define \[Q:=\bigcup_{\substack{z\in \Z^n \\ 0\leq \beta(z)\leq A_1}}\gamma(z)\big(B_Y(1,B_0)\cap H\big).\] That is, we take a ball in $\Gamma(G,Y)$ intersected with $H$ and close it under the set of elements of $\Z^n$ with small positive image under $\beta$ to obtain the set $Q$.  If $A_1$ happens to be negative then we take simply $Q=B_Y(1,B_0)\cap H$, and the reader may check  that the proof given below goes through with some simplifications.

Let us check that $Q$ is confining under $\gamma$ with respect to the homomorphism $\beta$. 

\begin{itemize}
\item First we check that if $\beta(z)\geq 0$ then $\gamma(z)(Q)\subseteq Q$. Let $z$ be such an element of $\Z^n$. An element of $Q$ has the form $\gamma(w)(h)$ where $h\in B_Y(1,B_0)\cap H$ and $w\in \Z^n$ with $0\leq \beta(w)\leq A_1$. We have $\beta(zw)=\beta(z)+\beta(w)\geq \beta(w)$. If $\beta(zw)\leq A_1$, then we have $\gamma(z)(\gamma(w)(h))=\gamma(zw)(h)\in Q$ by definition. Otherwise we have $\beta(zw)\geq A_1$. Hence $\beta((zw)^{-1})\leq -A_1$, and  our choice of $A_1$  
ensures that \[B_0\geq d_Y\big(h(zw)^{-1},(zw)^{-1}\big)=d_Y\big((zw)h(zw)^{-1},1\big)=d_Y\big(\gamma(zw)(h),1\big).\] Thus, $\gamma(zw)(h)\in B_Y(1,B_0)\cap H\subseteq Q$. 

\item Now let $h\in H$ be arbitrary.  We want to show that $\gamma(z)(h)\in Q$ for some $z\in \Z^n$.  Since $\beta$ is unbounded, there exists $z\in \Z^n$ satisfying $\beta(z)\geq A_0(d_Y(h,1))$, so that $\beta(z^{-1})\leq -A_0(d_Y(h,1))$. 
Hence by Claim \ref{claim:smalltransimprovement} we have \[B_0\geq d_Y\big(hz^{-1},z^{-1}\big)=d_Y\big(zhz^{-1},1\big)=d_Y\big(\gamma(z)(h),1\big).\] Thus we have $\gamma(z)(h)\in B_Y(1,B_0)\cap H\subseteq Q$.

\item Finally, we need to show that $\gamma(z)(Q\cdot Q)\subseteq Q$ for some $z\in \Z^n$. To see this, we first find a bound on the word length of elements of $Q$. An element of $Q$ has the form $\gamma(z)(h)=zhz^{-1}$ for some $z\in \Z^n$ with $0\leq \beta(z)\leq A_1$. By construction, the element $h$ has word length in $Y$ bounded by $B_0$. The element $z$ also has bounded word length. To see this, we first apply Lemma \ref{lem:abelianaxisdiam}, which shows that $d_Y\big(zc(0),c(-\beta(z))\big)\leq \consttwo$. Then, by the triangle inequality, \[d_Y\big(zc(0),c(0)\big)\leq d_Y\big(zc(0),c(-\beta(z))\big)+d\big(c(-\beta(z)),c(0)\big)\leq \consttwo+\beta(z)+20\delta\leq \consttwo+A_1+20\delta.\] Another application of the triangle inequality yields \[d_Y(z,1)\leq \consttwo+A_1+20\delta+2d_Y(x,1).\] Finally, this gives us a bound on the word length of $\gamma(z)(h)=zhz^{-1}$: \[d_Y(zhz^{-1},1)\leq B_0+2(\consttwo+A_1+20\delta+2d_Y(x,1)).\] Call this upper bound $\constthree$.

So far we have shown that $Q\subseteq B_Y(1,\constthree)$, from which it immediately follows that $Q\cdot Q\subseteq B_Y(1,2\constthree)$. 
Since there exists  $z\in \Z^n$ satisfying $\beta(z)\geq A_0(2\constthree)$,
  it follows from Claim \ref{claim:smalltransimprovement} that if 
 $h\in Q\cdot Q$, we have \[B_0 \geq d_Y(hz^{-1},z^{-1})=d_Y(zhz^{-1},1)=d_Y(\gamma(z)(h),1).\] That is, $\gamma(z)(Q\cdot Q)\subseteq B_Y(1,B_0) \cap H \subseteq Q$.
\end{itemize}

Now that we have constructed $Q$, we show that $\Gamma(G,Q\cup Z_\beta)$ is quasi-isometric to $\Gamma(G,Y)$, where the constant $C_\beta$ and the set $Z_\beta$ are chosen as in \eqref{eqn:Zrho}. This will complete the proof. To do this, we show that every element of $Q\cup Z_\beta$ has bounded word length with respect to the generating set $Y$ and vice versa. 

First we show that every element of $Q\cup Z_\beta$ has bounded word length in $Y$. We have already shown that $Q\subseteq B_Y(1,\constthree)$, so it remains to be shown that every element of $Z_\beta$ has bounded length in $Y$. For $z\in Z_\beta$ it follows exactly as in the third bullet point above that \[d_Y(z,1)\leq E+|\beta(z)|+20\delta+2d_Y(x,1)\leq E+C_\beta+20\delta+2d_Y(x,1).\]
Since this last quantity is independent of $z$, we have shown that every element of $Q\cup Z_\beta$ has bounded word length with respect to $Y$, as desired.

We now show that every element of $Y$ has bounded word length with respect to $Q\cup Z_\beta$. Consider an element $hz\in Y$ where $h\in H$ and $z\in \Z^n$. We will bound the word lengths of $h$ and $z$ with respect to $Q\cup Z_\beta$ separately. 

First we bound the word length of $z$ with respect to $Q\cup Z_\beta$. Note that we have $\beta(z)=\beta(hz)$. Moreover, by the definition of $\beta$, we have \[|\beta(hz)|\leq d_Y(x,hzx)\leq 2d_Y(1,x)+d_Y(1,hz)\leq 2d_Y(1,x)+1.\] Call this upper bound $\constfour$ so that  $|\beta(z)|\leq \constfour$. 
This allows us to bound the word length of $z$. We have $\beta(t_i)\neq 0$ for some $i$. Without loss of generality assume $\beta(t_1)\neq 0$. Then we may choose $n$ with \[|\beta(t_1^n)-\beta(z)|=|n\beta(t_1)-\beta(z)|\leq |\beta(t_1)|\leq C_\beta.\] Therefore $t_1^nz^{-1}\in Z_\beta$.  Moreover, $|n|$ is bounded by $\frac{\constfour}{|\beta(t_1)|}+1$ since $|\beta(z)|\leq \constfour$, and this proves that $z$ has word length at most \[|n|+1\leq \frac{\constfour}{|\beta(t_1)|}+2\] with respect to $Z_\beta$.

Now we bound the word length of $h$ with respect to $Q\cup Z_\beta$. We will first bound the word length of $h$ with respect to $Y$. Recall that we have already shown $|\beta(z)|\leq \constfour$, and thus another calculation identical to that of the third bullet point above yields \[d_Y(1,z)\leq E+|\beta(z)|+20\delta+2d_Y(x,1)\leq E+L+20\delta+2d_Y(x,1).\]
From this and the fact that $hz\in Y$, it follows that \[d_Y(1,h)\leq d_Y(1,hz)+d_Y(hz,h)\leq 1+ 2d_Y(x,1)+\consttwo+\constfour+20\delta.\]
 Denote by $M$ this upper bound on $d_Y(1,h)$.  By Claim \ref{claim:smalltransimprovement}, if $g \in B_Y(1, M)\cap H$ and $w\in \Z^n$ with $\beta(w)\leq -A_0(M)$, then $d_Y(gw,w)\leq B_0$.  In particular, we have $\beta(t_1^n)=n\beta(t_1)\leq -A_0(M)$ for some $n$ with $|n|\leq \frac{A_0(M)}{|\beta(t_1)|}+1$. For this value of $n$ we have $\gamma(t_1^n)(h)=t_1^nht_1^{-n}\in Q$. Therefore the word length of $h$ with respect to $Q\cup Z_\beta$ is at most \[2|n|+1\leq 2\frac{A_0(M)}{|\beta(t_1)|}+3.\] We have shown that every element of $Y$ has bounded word length with respect to $Q\cup Z_\beta$. We conclude that $\Gamma(G,Q\cup Z_\beta)$ is quasi-isometric to $\Gamma(G,Y)$ as desired.

It remains to show that $Q$ is \emph{strictly} confining exactly when $G\curvearrowright \Gamma(G,Y)$ is quasi-parabolic.  Since $\Gamma(G,Y)$ is quasi-isometric to $\Gamma(G,Q\cup Z_\beta)$, it suffices to show that $Q$ is strictly confining exactly when $G\curvearrowright \Gamma(G,Q\cup Z_\beta)$ is quasi-parabolic.  

If $Q$ is strictly confining, then it follows from Lemma \ref{focal} that $G \acts \Ga(G, Q \cup Z_\beta)$ is quasi-parabolic.  Conversely, suppose that $G \acts \Ga(G,Y)$ is quasi-parabolic and, towards a  contradiction, that $Q$ is confining but not strictly confining.  That is, suppose that for every $z \in \Z^n$ with $\rho(z) \geq 0$, we have that $\gamma(z)(Q) = Q$. It then follows that $Q = \gamma(z^{-1})(Q)$ for any such $z$, as well. By Definition \ref{def:confining}(b), we see that it thus follows that $Q =H$.  But since $\Gamma(\Z^n , Z_\beta)$ is a quasi-line, it follows that $\Gamma(G, Q \cup Z_\beta) = \Gamma(G, H \cup Z_\beta)$ is also a quasi-line, which contradicts the assumption that $G \acts \Ga(G,Y)$ is quasi-parabolic. Hence $Q$ must be strictly confining.
\end{proof}

We now introduce an equivalence relation on homomorphisms from $\Z^n$ to $\R$. 
\begin{defn}
We say two homomorphisms $\rho,\rho'\colon \Z^n\to \R$ are \emph{equivalent}, and write $\rho\sim\rho'$, if there exists a constant $j\in \R_{> 0}$ such that $\rho(z)=j\rho'(z)$ for all $z\in \Z^n$. 
\end{defn}
 As shown in Lemma \ref{lem:multiplekernel}, requiring that $j$ is positive  ensures that $\rho$ and $\rho'$ have not only  the same kernel, but also the same half space with positive image.  The following lemma shows that  a subset of $H$ is strictly confining with respect to at most one equivalence class of homomorphisms.   
 
 \begin{lem}\label{lem:notsim}
Suppose $Q\subsetneq H$ is strictly confining under $\gamma$ with respect to $\rho$. If $\rho'\sim\rho$, then $Q$ is strictly confining with respect to $\rho'$. Moreover, $Q$ is not confining under $\gamma$ with respect to $\rho''$ for any $\rho''\not\sim\rho$. 
 \end{lem}
 
 \begin{proof}
By Lemma \ref{lem:multiplekernel}, $\rho$ is completely determined up to scaling by the kernel of its extension to $\R^n$, which we (by an abuse of notation) denote $\ker\rho$.  We will show that $\ker\rho$ is completely determined by the strictly confining subset $Q$.  This will prove both statements.

 The kernel $\ker\rho$ is a linear codimension-one subspace of $\R^n$ that divides $\R^n\setminus \ker\rho$ into two half spaces, $H_1$ and $H_2$, in the following way: for any $z \in \Z^n$, $z\in H_1$ if and only if $\rho(z)>0$,  and $z\in H_2$ if and only if $\rho(z)<0$.
 
We claim that for any $z\in\Z^n$, we have
$\gamma(z)(Q)\subsetneq Q$ if and only if $z\in H_1$, and $\gamma(z)(Q)\supsetneq Q$ if and only if $z\in H_2$.
We will show that the first statement holds; the proof of the second statement is similar.
Since $Q$ is strictly confining, there exists some $w\in \Z^n$ such that $\gamma(w)(Q)\subsetneq Q$.  Suppose $z\in H_1$, so that  $\rho(z)>0$.  If $\rho(z)\geq \rho(w)$, then by writing $z=(zw^{-1})w$, we have 
\[\gamma(z)(Q)=\gamma(w)\gamma(zw^{-1})(Q)\subseteq\gamma(w)(Q)\subsetneq Q.\]
  On the other hand, if $0<\rho(z)<\rho(w)$, then $\rho(z^m)>\rho(w)$ for some $m$,  and then by the argument above, we have 
\[\gamma(z)^m(Q)=\gamma(z^m)(Q)\subsetneq Q.\] 
But if $\gamma(z)(Q)=Q$, then  $\gamma(z)^m(Q)= Q$, which is a contradiction.  Therefore $\gamma(z)(Q)\subsetneq Q$.

Conversely, suppose $\gamma(z)(Q)\subsetneq Q$ for some $z\in \Z^n$.  If $z\notin H_1$, then $\rho(z)\leq 0$, and so $\rho(z^{-1})\geq 0$.  It follows that $\gamma(z^{-1})(Q)\subseteq Q$.  Thus  $Q=\gamma(zz^{-1})(Q)=\gamma(z)\gamma(z^{-1})(Q)\subseteq \gamma(z)(Q)\subsetneq Q$, which is a contradiction.
 \end{proof}

\section{An example: $\Z[\frac1k]\rtimes \Z^n$}
Fix a natural number $k\geq 2$ which is not a power of a prime number. Let $k=p_1^{m_1}\cdots p_n^{m_n}$ be the prime factorization of $k$.  Let  $G_k=\Z[\frac1k]\rtimes_\gamma \Z^n$, where the homomorphism $\gamma\colon \Z^n \to \operatorname{Aut}(\Z[\frac{1}{k}])$ is defined as follows. Let $t_1,\ldots,t_n$ be a basis for $\Z^n$ as a free abelian group. For any $h\in\Z[\frac{1}{k}]$, we define $\gamma(t_i)(h)=t_iht_i^{-1}$ to be equal to $p_i^{m_i} h$ as an element of $\Z[\frac{1}{k}]$, so that $\gamma(t_i)$ acts on $\Z[\frac1k]$ as multiplication by  $p_i^{m_i}$.  Thus $G_k$ is the group with presentation \[G_k = \left\langle a, t_1,\ldots,t_n \mid [t_i,t_j]=1, \ t_iat_i^{-1}=a^{p_i^{m_i}} \text{ for all } i,j\right\rangle\] and is a higher-dimensional analog of a solvable Baumslag-Solitar group. Here $a$ corresponds to the normal generator 1 of $\Z[\frac1k]$. The relation $t_1at_1^{-1}=a^{p_1^{m_1}}$ implies that for any pseudocharacter $\beta\colon G_k\to \R$, \[\beta(a)=\beta(t_1at_1^{-1})=p_1^{m_1}\beta(a),\] and thus $\beta(a)=0$. Therefore any pseudocharacter vanishes on  $\Z[\frac1k]$; this applies in particular to any Busemann pseudocharacter.  Moreover, as before, $\beta$ turns out to be a homomorphism.

In addition to the standard representation of elements of $\Z[\frac1k]$ as Laurent polynomials in $k$, we also represent elements of $\Z[\frac1k]$ by their base $k$ expansions.  For example,  $\frac1k=0.1$ while $k+\frac1k+\frac{1}{k^4}=10.1001$. In general for $c\in \Z[\frac1k]$ we may write \[c=\pm c_r\ldots c_0.c_{-1}\ldots c_{-s}\] where each digit $c_i\in \{0,\ldots,k-1\}$. We will switch between these representations interchangeably. We note that in the base $k$ representation, multiplying and dividing by $k$ shifts the decimal point one place to the right and left, respectively.

The goal of this section is to prove Theorem \ref{thm:Z1k}, which we restate for the convenience of the reader.
 
\Gk*

 Our strategy for classifying the quasi-parabolic structures reduces the problem to the classification of quasi-parabolic structures of the solvable Baumslag-Solitar group $BS(1,k)$, which was given by the first and third authors in \cite{AR}. The Baumslag-Solitar group $BS(1,k)$ is the group with presentation $BS(1,k) = \langle a,t \mid tat^{-1}=a^k\rangle$.
It is isomorphic to the semi-direct product $\Z[\frac{1}{k}]\rtimes_\alpha \Z$, where the automorphism $\alpha$ acts as multiplication by $k$. We will identify $BS(1,k)$ with this semi-direct product in all that follows.

The following theorem from \cite{AR} classifies the subsets of $\Z[\frac{1}{k}]$ which are confining under the action of $\alpha$. Let $s$ generate the factor $\Z$ so that $BS(1,k)\simeq \Z[\frac{1}{k}]\rtimes_\alpha \langle s\rangle$.  We say a divisor $l$ of $k=p_1^{m_1}\ldots p_n^{m_n}$ is \emph{full} if $p_i^{m_i}$ divides $l$ whenever $p_i$ divides $l$. For example, the full divisors of 12 are 1, 4, 3, and 12.

\begin{prop}[{\cite[Theorem~1.1~\&~Proposition~3.12]{AR}}]\label{thm:BS1k}
If $C\subseteq\Z[\frac{1}{k}]$ is confining under the action of $\alpha$, then there is a full divisor $l$ of $k$ such that 
\[
[C\cup \{s^{\pm 1}\}]= \left[\Z\left[\frac{1}{l}\right]\cup \{s^{\pm 1}\}\right]\]
as elements of $ \mathcal H_{qp}(BS(1,k))$.
If $C\subseteq\Z[\frac{1}{k}]$ is confining under the action of $\alpha^{-1}$, then either $[C\cup\{s^{\pm 1}\}]=[\Z[\frac{1}{k}]\cup\{s^{\pm 1}\}]$ or $[C\cup\{s^{\pm 1}\}]=[C_-\cup\{s^{\pm 1}\}]$ as  elements of $ \mathcal H_{qp}(BS(1,k))$, where 
\[
C_-=\left\{c\in \Z\left[\frac{1}{k}\right] \Big\vert \ c=\pm 0.c_{-1}c_{-2}\ldots c_{-m} \textrm{ for some $m\in \mathbb N$}\right\}. 
\]
\end{prop}
 
\noindent Stated another way, $C_-$ is the unit ball of $\R$ intersected with $\Z[\frac1k]$.
 
 \subsection{Quasi-parabolic structures}
 \label{section:QPofGk}
In this subsection we describe the quasi-parabolic structures of $G_k$.  In particular, we prove the following proposition.
 \begin{prop}\label{prop:qp}
$\mathcal H_{qp}(G_k)$ consists of exactly $n+1$  incomparable structures.
\end{prop}

 \subsubsection{Confining subsets of $\Z[\frac1k]$}

We will describe \emph{all} of the subsets of $\Z[\frac1k]$ which are confining under $\gamma$ with respect to some homomorphism $\rho : \Z^n \to \R$. We note that the subset $Q=\Z[\frac1k]$ is  confining  but not strictly confining under $\gamma$ with respect to all choices of $\rho$ and corresponds to a lineal structure by Lemma \ref{focal}.

  To prove Proposition \ref{prop:qp}, we will show that there are exactly $n+1$ choices for $\rho$ with respect to which there are any  subsets of $\Z[\frac1k]$ that are \emph{strictly} confining under $\gamma$.  By Theorem \ref{thm:main}, each such subset corresponds to a quasi-parabolic structure.  For each of these $n+1$ choices of $\rho$, we will show there is exactly one such structure.

  We begin with a preliminary lemma, which is the analogue of \cite[Lemma~3.2]{AR} and \cite[Lemma 4.9]{Bal}. The proof follows the same lines with a few modifications.   Note that this lemma holds for an arbitrary group $H\rtimes_\gamma\Z^n$.
  
\begin{lem}\label{lem:AddS}
Consider a group $H\rtimes_\gamma \Z^n$, and fix a homomorphism $\rho\colon \Z^n\to\mathbb R$.  Suppose $Q$ is a symmetric subset of $H$ which is confining under $\gamma$ with respect to $\rho$.  Let $S$ be a symmetric subset of $H$ such that there exists $z_0\in\Z^n$ with $\rho(z_0)\geq 0$ such that $\gamma(z_0)(g)\in Q$ for all $g\in S$.  Then 
 \[
 \overline Q=Q\cup\bigcup_{\rho(z)\geq 0}\gamma(z)(S)
 \]
 is confining under $\gamma$ with respect to $\rho$, and 
 \[
 [Q\cup Z_\rho]=[\overline Q\cup Z_\rho].
 \]
 \end{lem}
 
 \begin{proof}
 Conditions (a) and (b) of Definition \ref{def:confining} are clear.  To show that (c) holds, 
note that for any $z\in\Z^n$ with $\rho(z)\geq 0$ and any $g\in S$, 
\[
\gamma(z_0)(\gamma(z)(g))=\gamma(z)(\gamma(z_0)(g))\in \gamma(z)(Q)\subseteq Q.
\] 
We also have $\gamma(z_0)(g)\in Q $ for any $g\in Q$. Hence, $\gamma(z_0)(\overline Q)\subseteq Q$.  Let $g,h\in\overline Q$, and fix $z_1\in\Z^n$ such that $\gamma(z_1)(Q\cdot Q)\subseteq Q$.  Then 
\[
\gamma(z_0z_1)(gh)=\gamma(z_1)\big(\gamma(z_0)(g)\gamma(z_0)(h)\big)\in \gamma(z_1)(Q\cdot Q)\subseteq  Q\subseteq \overline Q.
\]
Therefore (c) holds with the element $z_0z_1$.

To see that $[Q\cup Z_\rho]=[\overline Q\cup Z_\rho]$, note first that  $[\overline Q\cup Z_\rho]\preceq [Q\cup Z_\rho]$ since $Q\subseteq \overline Q$. For the other direction, let $\rho(z_0)=K_0$.  Then for $z\in\Z^n$ with $\rho(z)\geq K_0$, we have $\gamma(z)(S)\subseteq Q$, since $\gamma(z_0)(S)\subseteq Q$. Hence we may also write
\[
\overline Q=Q\cup \bigcup_{0\leq \rho(z)<K_0}\gamma(z)(S).
\]

Let $g \in \overline{Q}$. If $g \in Q$, then $\|g\|_{Q \cup Z_\rho} \leq 1$. As above, if  $g  = \gamma(z)(s)$ for some $s \in S$ and $0 \leq \rho(z) < K_0$, then  $\gamma(z_0z)(s) \in Q$. Thus $g = \gamma(z)(s) \in \gamma(z_0^{-1})(Q)$.  From this one can see that  $\|g\|_{Q \cup Z_\rho} \leq 2||z_0||_{Z_\rho} +1$, which is a constant independent of $g$.
In other words, any element of $\overline Q\cup Z_\rho$ has uniformly bounded word length with respect to $Q\cup Z_\rho$, and so $[Q\cup Z_\rho]\preceq [\overline Q\cup Z_\rho]$.
 \end{proof}
 
 The following lemma is proven exactly as in the proof of \cite[Lemma~3.3]{AR}, with $p_i^{a_i}$ playing the role of $n$ and $t_i$ playing the role of $t$. We refer the reader to \cite{AR} for the proof. Here $Q\cup \Z$ denotes the union of subsets of $\Z[\frac{1}{k}]$. 
 
 \begin{lem}\label{lem:ZsubsetQ}
 If $\rho(t_i)>0$ for some $i$, then $[Q\cup Z_\rho]=[(Q\cup\Z)\cup Z_\rho]$.
 \end{lem}
 
 \noindent In other words, the subring $\Z$ of $\Z[\frac1k]$ has bounded word length in $Q\cup Z_\rho$.
 
The next lemma describes $n$ distinct homomorphisms $ \Z^n\to \R$, and, for each homomorphism, identifies a subset of $\Z[\frac1k]$ which is strictly confining under $\gamma$ with respect to it. Define $k_i=\frac{k}{p_i^{m_i}}=p_1^{m_1} \cdots \widehat{p_i^{m_i}} \cdots p_n^{m_n}$, where $\widehat{\cdot}$ indicates that we omit that factor from the product. 

 \begin{lem}\label{lem:Q_i}
 Let $\rho^+_i\colon\Z^n\to\R$ be projection to the $i$-th factor of $\Z^n$ with respect to the basis $\{t_1, t_2,\cdots, t_n\}$, and let $Q_i = \Z\left[\frac{1}{k_i}\right]$.
 Then $Q_i$ is strictly confining under $\gamma$ with respect to $\rho_i^+$.
 \end{lem}

 \begin{proof}

 Fix some $1\leq i\leq n$.    We will first show that $Q_i$ is confining under $\gamma$ with respect to $\rho_i^+$. Fix $z\in\Z^n$ with $\rho_i^+(z)\geq 0$, so that $z=t_1^{b_1}\ldots t_n^{b_n}$, where $b_i\geq 0$. Note that $Q_i$ is closed under multiplication by integers. Since $p_1^{b_1m_1}\cdots \widehat{p_i^{b_im_i}}\cdots p_n^{b_nm_n}$ is an integer times a (possibly negative) power of $k_i$, it follows that for any $q\in Q_i=\Z[\frac{1}{k_i}]$ we have \[\gamma(z)(q) = p_i^{b_im_i} \left(p_1^{b_1m_1} \cdots \widehat{p_i^{b_im_i}} \cdots p_n^{b_nm_n} q\right) \in p_i^{b_im_i} \Z\left[\frac{1}{k_i}\right].\] 
Since $b_i\geq 0$, we have $p_i^{b_im_i}\in \Z$, and so  $\gamma(z)(q)\in Q_i=\Z[\frac{1}{k_i}]$, as desired.

Let $h=\pm h_m\dots h_0.h_{-1}h_{-2}\dots h_{-\ell}\in \Z[\frac1k]$.  Since  $\rho^+_i(t_1\ldots t_n)=1$, we have $\rho^+_i((t_1\ldots t_n)^\ell)=\ell>0$.  Let $z=(t_1\ldots t_n)^\ell$.  Then 
\[
\gamma(z)(h)=\gamma((t_1\ldots t_n)^\ell)(h)=k^\ell h=\pm h_m\dots h_0h_{-1}h_{-2}\dots h_{-\ell}\in\Z\subseteq Q_i,
\] 
 and Definition \ref{def:confining}(b) is satisfied.
 To see that Definition \ref{def:confining}(c) holds, notice that $Q_i$ is closed under addition, and so we can take $z_0$ to be the identity of $\Z^n$.

Finally, $Q_i$ is also \emph{strictly} confining with respect to $\rho_i^+$:   $1\in Q_i$, but $1\notin \gamma(t_i)(Q_i)$ since $\gamma(t_i^{-1})(1)=\frac{1}{p_i^{m_i}}\notin Q_i$. Thus $\gamma(t_i)(Q_i)\subsetneq Q_i$.
 \end{proof}
  
We now describe one additional homomorphism $ \Z^n\to\R$ and a subset which is strictly confining under $\gamma$ with respect to this homomorphism.  We will show that this confining subset, along with the $n$ subsets constructed in Lemma \ref{lem:Q_i} give rise to the $n+1$  quasi-parabolic structures from Proposition \ref{prop:qp}.
 \begin{lem}\label{lem:Q_-}
 Let $\rho_-\colon\Z^n\to\R$ be given by $\rho_-(t_i)=-m_i\log p_i$ for each $i$, and let 
 \[
 Q_-=\left\{x\in \Z\left[\frac1k\right]\Big\vert \ x=\pm 0.x_{-1}x_{-2}\dots x_{-m} \textrm{ for some }m\in \N\right\}\subsetneq \Z\left[\frac1k\right]
 \]
(that is, $Q_-$ is the unit ball in $\R$ intersected with $\Z[\frac{1}{k}]$). Then $Q_-$ is strictly confining under $\gamma$ with respect to $\rho_-$.
 \end{lem}
 
 \begin{proof}
 Suppose $z=t_1^{a_1}\ldots t_n^{a_n}\in\Z^n$ satisfies $\rho_-(z)\geq 0$.  Then $-a_1m_1\log p_1-\cdots -a_nm_n\log p_n\geq0$, and so $p_1^{a_1m_1}\cdots p_n^{a_nm_n}\leq 1$.  For $q=\pm0.x_{-1}x_{-2}\ldots x_{-m}\in Q_-$, we have
 \[
 \gamma(z)(q)=\pm (p_1^{a_1m_1}\cdots p_n^{a_nm_n})(0.x_{-1}x_{-2}\dots x_{-m}).
 \]
 In particular, we have a bound on absolute values $|\gamma(z)(q)|\leq |q|$, and thus $\gamma(z)(q)\in Q_-$.  Therefore Definition \ref{def:confining}(a) holds.  Next, let $h=\pm h_m\dots h_0.h_{-1}h_{-2}\dots h_{-\ell}$ be an arbitrary element of $ \Z[\frac1k]$.  
 Taking $z=(t_1^{-1}\ldots t_n^{-1})^m$, we see that 
\[
\gamma(z)(h)=\gamma((t_1^{-1}\ldots t_n^{-1})^m)(h)=k^{-m} h=\pm 0.h_m\dots h_0h_{-1}h_{-2}\dots h_{-\ell}\in Q_- ,
\] 
 and Definition \ref{def:confining}(b) is satisfied.  Finally, let $x,y\in Q_-$.  Then we have 
 \[
 x+y=\pm a_0.a_{-1}\ldots a_{-\ell}
 \]
 where each $a_i\in\{0,\dots, k-1\}$. Let $z_0=t_1^{-1}\ldots t_n^{-1}$.  Then $\rho_-(z_0)>0$, and 
 \[
 \gamma(z_0)(x+y)=\pm k^{-1}\left(  a_0.a_{-1}\ldots a_{-\ell}\right)=\pm0.a_0a_{-1}\ldots a_{-\ell}\in Q_-,
 \]
 and so Definition \ref{def:confining}(c) holds with this choice of $z_0$.  
 
 It remains to show that $Q_-$ is \emph{strictly} confining. 
 Fix any $z=t_1^{a_1}\ldots t_n^{a_n} \in \Z^n$ with $p_1^{m_1a_1}\cdots p_n^{m_na_n}<\frac{1}{k}$. We have $\rho_-(z)=-\log(p_1^{m_1a_1}\cdots p_n^{m_na_n})>0$ and
 \[
 \gamma(z^{-1})(0.1)=(p_1^{m_1a_1}\cdots p_n^{m_na_n})^{-1}(0.1)>1.
 \]
It follows that $\gamma(z^{-1})(0.1)\not\in Q_-$ and thus $0.1 \notin \gamma(z)(Q_-)$.  Since $0.1\in Q_-$, we conclude that $Q_-$ is strictly confining under $\gamma$ with respect to $\rho_-$.
  \end{proof} 

In the following series of lemmas, we will show that given any homomorphism $\rho\colon\Z^n\to\R$, if  there is a subset of $\Z[\frac{1}{k}]$ which is strictly confining under $\gamma$ with respect to $\rho$, then $\rho$ must be equivalent to $\rho_-$ or $\rho_i$ for some $i$, and the strictly confining subset must give rise to the same quasi-parabolic structure on $G_k$ as $Q_-$ or $Q_i$, respectively.

Fix a homomorphism $\rho\colon\Z^n\to \R$.  We will consider three separate cases, depending on whether $\rho(t_1\ldots t_n)$ is positive, negative, or zero.  Each case is dealt with in a separate lemma.

We first show that if $\rho(t_1\ldots t_n)=0$, there are \emph{no} subsets of $\Z[\frac{1}{k}]$ which are strictly confining under $\gamma$ with respect to $\rho$.
\begin{lem}\label{lem:rho=0}
Suppose $\rho(t_1\ldots t_n)=0$.  If $Q\subseteq \Z[\frac1k]$ is confining under $\gamma$ with respect to $\rho$, then 
\[
[Q\cup Z_\rho]=\left[\Z\left[\frac1k\right]\cup Z_\rho\right].
\]
\end{lem}

\begin{proof}
Suppose $Q\subseteq\Z[\frac1k]$ is confining under $\gamma$ with respect to $\rho$. Since $\rho$ is not identically equal to zero, it must be the case that $\rho(t_i)>0$ for some $i$. Thus, by Lemma \ref{lem:ZsubsetQ}, $[Q\cup Z_\rho]=[(Q\cup\Z) \cup Z_\rho]$, where $\Z$ denotes the subring of $\Z[\frac{1}{k}]$. This implies, in particular, that $\Z$ has bounded word length with respect to $Q\cup Z_\rho$. Recall that $k = p_1^{m_1}\ldots p_n^{m_n}$ and that the action of each $t_i$ is by multiplication by $p_i^{m_i}$. Since $\rho(t_1 \ldots t_n)=0$, we have that $\rho(t_1^{-1}\ldots t_n^{-1})=0$ as well, and so by Definition \ref{def:confining}(a) $Q$ is closed under multiplication by $\frac{1}{k}$. Since $\Z$ has bounded word length with respect to $Q\cup Z_\rho$ and $Q$ is closed under multiplication by $\frac{1}{k}$, we see that all of $\Z[\frac{1}{k}]$ has bounded length with respect to $Q\cup Z_\rho$ as well. 
This completes the proof. 
\end{proof}

We next consider the case  $\rho(t_1 \ldots t_n)>0$ and relate confining subsets of $\Z[\frac{1}{k}]$ in $G_k$ to confining subsets of $\Z[\frac{1}{k}]$ in the Baumslag-Solitar group $BS(1,k)$.
\begin{lem}\label{lem:QBS}

Let $Q\subseteq \Z[\frac1k]$ be confining under $\gamma$ with respect to $\rho$.  View $Q\subseteq\Z[\frac1k]$ as a subset of the Baumslag-Solitar group $BS(1,k)=\Z[\frac1k]\rtimes_\alpha \Z$.  If $\rho(t_1 \ldots t_n)> 0$, then $Q$ is confining under the action of $\alpha$.
\end{lem}

\begin{proof}
We will show that $Q$ satisfies the conditions of Definition \ref{genconfine}.   
Recall that $BS(1,k)=\Z[\frac1k]\rtimes_\alpha\langle s\rangle$.  Then $\alpha$ and $\gamma(t_1 \ldots t_n)$ both act on $\Z[\frac1k]$ as multiplication by $k$, and so $\alpha(Q)=\gamma(t_1\ldots t_n)(Q)$.  Since $\rho(t_1\ldots t_n)> 0$, we have $\alpha(Q)\subseteq Q$, and  Definition \ref{genconfine}(a) is satisfied.  Moreover, for any $u\in \Z[\frac1k]$ and $z\in \Z^n$, we have $\gamma(z)(u)\in Q$ whenever $\rho(z)$ is sufficiently large. In particular $\alpha^a(u)=\gamma((t_1\ldots t_n)^a)(u)\in Q$ for $a\in \Z$ sufficiently large, and  Definition \ref{genconfine}(b) is satisfied.  Finally, since  $\rho(t_1\ldots t_n)> 0$, there is some constant $b_0\in \Z$ such that $\rho((t_1\ldots t_n)^{b_0})>\rho(z_0)$, where $z_0$ is as in Definition \ref{def:confining}(c).  Thus $\alpha^{b_0}(Q+Q)=\gamma\big((t_1\ldots t_n)^{b_0}\big)(Q+Q)\subseteq Q$, and  Definition \ref{genconfine}(c) is satisfied.  Therefore, $Q$ is confining under $\alpha$.
\end{proof}

\begin{lem}\label{lem:rho>0}
Suppose $\rho(t_1\ldots t_n)>0$ and $Q\subseteq \Z[\frac1k]$ is confining under $\gamma$ with respect to $\rho$.  
If $Q$ is strictly confining, then for some $i$ we have
\[
[Q\cup Z_\rho]=[Q_i\cup Z_\rho]
\] and $\rho\sim\rho^+_i$.  Otherwise, \[
[Q\cup Z_\rho]=\left [\Z\left[\frac1k\right ]\cup Z_\rho \right ].
\] 
\end{lem}

\begin{proof}
Let $Q$ be confining under $\gamma$ with respect to $\rho$, where $\rho(t_1 \ldots t_n)>0$, as in the statement of the lemma.  
Our assumption that $\rho(t_1 \ldots t_n)>0$ ensures that there exists $1\leq i\leq n$  with $\rho(t_i)>0$. We will show that when $Q$ is strictly confining (equivalently, when $[Q\cup Z_\rho]\neq \left[ \Z\left[\frac{1}{k}\right] \cup Z_\rho\right]$), there is  a unique such index $i$, and $[Q\cup Z_\rho]=[Q_i\cup Z_\rho]$ for this $i$.    

By Lemma \ref{lem:QBS}, we see that $Q$ is a confining subset of $\Z[\frac{1}{k}]$ under the automorphism $\alpha$ given by multiplication by $k$. Hence   we may consider the Baumslag-Solitar group $BS(1,k)=\Z[\frac{1}{k}]\rtimes_\alpha \langle s \rangle$, and by Proposition \ref{thm:BS1k} we have that $[Q\cup \{s^{\pm 1}\}]=[\Z[\frac{1}{l}]\cup \{s^{\pm 1}\}]$, as generating sets of $BS(1,k)$, for some full divisor $l$ of $k$.

It follows that there exists a positive integer $N$ such that every element of $Q$ has word length at most $N$ in the generating set $\Z[\frac{1}{l}]\cup \{s^{\pm 1}\}$. We claim additionally that $\alpha^N(Q) = k^N Q \subseteq \Z[\frac{1}{l}]$. To see this, write an element $g\in Q$ as a word $g=g_1\ldots g_r$ in the generating set $\Z[\frac{1}{l}]\cup \{s^{\pm 1}\}$. Since $\alpha(\Z[\frac1l])=s\Z[\frac1l]s^{-1}\subseteq \Z[\frac1l]$, we may move any $g_i$ which is equal to $s$ to the right and any $g_i$ which is equal to $s^{-1}$ to the left, keeping the length of the word unchanged. The result is a geodesic word \[g=s^{-t} h_1 \ldots h_v s^u\] where $h_i \in \Z[\frac{1}{l}]$ for all $i$ and $t,u\geq 0$. Since $r\leq N$ we also have $t,u\leq r \leq N$. Moreover, since $\Z[\frac{1}{l}]$ is a subgroup of $\Z[\frac{1}{k}]$, we have $h_1\ldots h_v=h\in \Z[\frac{1}{l}]$. Finally, since $g$ is contained in $\Z[\frac{1}{k}]$ we have $t=u$. Thus \[g=s^{-t}hs^t=\alpha^{-t}(h),\] and this proves that $\alpha^t(g)=h\in \Z[\frac{1}{l}]$ with $t\leq N$.

We will next show that either $l=k$ or $l=k_i=p_1^{m_1} \ldots \widehat{p_i^{m_i}} \ldots p_n^{m_n}$, where $\widehat{\cdot}$ denotes omission of the corresponding factor. For any $j\neq i$, there exists $f>0$ sufficiently large that $\rho(t_i^f t_j^{-1})>0$. Consequently $Q$ is closed under $\gamma(t_i^ft_j^{-1})$, which is the automorphism given by multiplication by $p_i^{fm_i}/p_j^{m_j}$. For any $x\in Q\setminus \{0\}$ and any integer $r>0$, we have \[\frac{p_i^{rfm_i}}{p_j^{rm_j}}x\in Q,\] and thus 
\[ k^N \frac{p_i^{rfm_i}}{p_j^{rm_j}} x = \alpha^N\left(\frac{p_i^{rfm_i}}{p_j^{rm_j}}x\right) \in \Z\left[\frac{1}{l}\right].\] 
Writing this rational number as a reduced fraction we see that, as long as $r$ is sufficiently large, the denominator is divisible by $p_j$. Since any element of $\Z[\frac{1}{l}]$ has a divisor of a power of $l$ as its denominator when written as a reduced fraction, we must have that $p_j$ divides $l$. Finally, since $l$ is a \emph{full} divisor of $k$, this implies that $p_j^{m_j}$ divides $l$. Since this is true for \emph{any} $j\neq i$, it follows that $k_i$ divides $l$. Thus $l=k_i$ or $l=k$.

If there is some index $j\neq i$ such that $\rho(t_j)$ is also positive, then by running the above argument with $j$ in place of $i$, we see that $k_j$ must also divide $l$.  This implies that $l=k$.  In particular, the only way we can have $l=k_i$ is if $\rho(t_i)>0$ and $t_i$ is the \emph{unique} generator with positive image.

Regardless of whether $l=k_i$ or $l=k$, we have that $[Q\cup \{s^{\pm 1}\}]=[\Z[\frac{1}{l}]\cup \{s^{\pm 1}\}]$. This also proves $[Q\cup Z_\rho]=[\Z[\frac{1}{l}]\cup Z_\rho]$. To see this, note that every element of $\Z[\frac{1}{l}]$ has bounded word length in $Q\cup \{s^{\pm 1}\}$ and that $t_1\ldots t_n$, which acts by conjugation on $\Z[\frac{1}{l}]$ in the same way as $s$, has finite word length in $Z_\rho$. Thus $\Z[\frac{1}{l}]$ has bounded word length in $Q\cup Z_\rho$. Since elements of $Z_\rho$ trivially have bounded word length in $Q\cup Z_\rho$, all of $\Z[\frac{1}{l}]\cup Z_\rho$ has bounded word length in $Q\cup Z_\rho$. The fact that $Q\cup Z_\rho$ has bounded word length in $\Z[\frac{1}{l}]\cup Z_\rho$ is also trivial.

It remains to show that if $l=k_i$ then $\rho\sim \rho_i^+$. Recall that $i$ is the unique index with $\rho(t_i)>0$.  Suppose towards a contradiction that $\rho(t_j)<0$ for some $j\neq i$.  Since $Q$ is confining under $\gamma$ with respect to $\rho$, for any $h\in \Z[\frac{1}{k}]$, there is some $z\in \Z^n$  such that $\gamma(z)(h)\in Q$. In particular, for each positive integer $L$, there is such an element $z\in \Z^n$ so that $\gamma(z)\left(1/p_i^{Lm_i}\right)\in Q$.  Since $\rho(t_j^{-1})>0$ we have $\rho(t_j^{-M})>\rho(z)$ for all sufficiently large $M$. Thus   
\[
\gamma\left(t_j^{-M}\right)\left(\frac{1}{p_i^{Lm_i}}\right)=p_j^{-Mm_j}p_i^{-Lm_i}\in Q,
\]
and so the word length of $p_j^{-Mm_j}p_i^{-Lm_i}$  with respect to $\Z[\frac{1}{k_i}]\cup \{s^{\pm 1}\}$ is bounded independently of the choice of $L$, as long as $M$ is large enough. Taking $L$ arbitrarily large leads to a contradiction.
\end{proof}

Finally, we turn our attention to homomorphisms satisfying $\rho(t_1\ldots t_n)<0$ and again relate confining subsets of $\Z[\frac{1}{k}]$ in $G_k$ to confining subsets of $\Z[\frac{1}{k}]$ in $BS(1,k)$.
\begin{lem}\label{lem:QBS-}
Let $Q\subseteq \Z[\frac1k]$ be confining under $\gamma$ with respect to $\rho$.  View $Q\subseteq\Z[\frac1k]$ as a subset of the Baumslag-Solitar group $BS(1,k)=\Z[\frac1k]\rtimes_\alpha \Z$.  If $\rho(t_1 \ldots t_n)< 0$, then $Q$ is confining under $\alpha^{-1}$.\end{lem}  

\begin{proof}
The proof follows the same lines as the proof of Lemma \ref{lem:QBS}. To show that Definition \ref{genconfine}(a) holds, we note that $\rho(t_1^{-1}\ldots t_n^{-1})>0$, and so $\alpha^{-1}(Q)=\gamma(t_1^{-1}\ldots t_n^{-1})(Q)\subseteq Q$.   For Definition \ref{genconfine}(b), let $u\in \Z[\frac{1}{k}]$. By Definition \ref{def:confining}(b), $\gamma(z)(u)\in Q$ for all $z\in \Z^n$ with $\rho(z)$ sufficiently large.  If $a\in \Z$ is chosen so that $\rho((t_1^{-1}\ldots t_n^{-1})^a)$ is sufficiently large, then $\alpha^{-a}(u)=\gamma((t_1^{-1}\ldots t_n^{-1})^a)(u)\in Q$. The proof that Definition \ref{genconfine}(c) holds is analogous.
\end{proof}

\begin{lem}\label{lem:rho<0}
Suppose $\rho(t_1\ldots t_n)<0$.  If $Q\subseteq \Z[\frac1k]$ is strictly confining under $\gamma$ with respect to $\rho$, then 
\[
[Q\cup Z_\rho]=[Q_-\cup Z_\rho]
\]  and  $\rho \sim \rho_-$.
\end{lem}

\begin{proof}
The proof follows the same lines as the proof of Lemma \ref{lem:rho>0}.  Suppose $Q\subseteq\Z[\frac1k]$ is strictly confining under $\gamma$ with respect to $\rho$.  By Lemma \ref{lem:QBS-}, $Q$ is a subset which is confining under $\alpha^{-1}$, when viewed as a subset of $BS(1,k)=\Z[\frac1k]\rtimes_\alpha \Z$.  By Proposition \ref{thm:BS1k}, up to bounded word length, the subsets $A$ of $\Z[\frac1k]\subseteq BS(1,k)$ which are confining under $\alpha^{-1}$ are exactly $A=\Z[\frac1k]$ and $A=Q_-$, where here word length is measured in the generating set $A\cup \{s^{\pm1}\}$.  As we assume that $Q$ is strictly confining under $\gamma$ with respect to $\rho$, the set $Q$ is not within bounded word length of $\Z[\frac1k]$.  Therefore, $Q$ is within bounded distance of $Q_-$ when word length is measured in the generating set $Q_-\cup \{s^{\pm1}\}$. 

As in the proof of Lemma \ref{lem:rho>0},  we see that $Q_-$ has bounded word length with respect to the generating set $Q\cup Z_\rho$ and vice versa. Finally, we show that $\rho\sim \rho_-$. The proof is analogous to the proof of the corresponding fact in Lemma \ref{lem:rho>0}. If $\rho\not\sim \rho_-$, then there exists $z\in \Z^n$ with $\rho(z)>0$ but $\rho_-(z)< 0$. Writing $z=t_1^{a_1}\ldots t_n^{a_n}$ and $y=p_1^{a_1m_1}\cdots p_n^{a_nm_n}\in \Z[\frac1k]$, we have that $\gamma(z)$ acts as multiplication by $y$ and that $y>1$ (since $\rho_-(z)<0$). Choosing any $x\in Q\setminus \{0\}$, we have that $\gamma(z^i)(x)=y^ix\in Q$ for any $i\geq 0$, and thus $y^ix$ has bounded word length in $Q_-\cup \{s^{\pm 1}\}$ for each $i\geq 0$. For very large $i$, the number $y^ix$ is very large in absolute value, and therefore the word length of $y^ix$ in $Q_-\cup \{s^{\pm 1}\}$ is also very large. This is a contradiction.
\end{proof}

We are now ready to prove Proposition~\ref{prop:qp}.
\begin{proof}[Proof of Proposition~\ref{prop:qp}] Fix $[Y]\in\mathcal H_{qp}(G_k)$, and let $\beta$ be the associated Busemann pseudocharacter. As explained at the beginning of Section 4, we have that $\beta(\Z[\frac1k])=0$ and $\beta$ is a homomorphism.

  By Theorem \ref{prop:main}, there exists $Q\subseteq \Z[\frac1k]$ which is strictly confining under $\gamma$ with respect to $\beta$ and such that $[Q\cup Z_\beta]\sim [Y]$, where $Z_\beta$ is as in \eqref{eqn:Zrho}. Conversely, Theorem \ref{thm:main} shows that for each subset $Q\subseteq \Z[\frac1k]$ which is strictly confining under $\gamma$ with respect to some homomorphism $\rho\colon\Z^n\to\R$, we have $[Q\cup Z_\rho]\in\mathcal H_{qp}(G_k)$ and the Busemann pseudocharacter for this hyperbolic structure is equivalent to $\rho$.
  
  Therefore Theorems \ref{thm:main} and \ref{prop:main} show that it suffices to classify equivalence classes of generating sets of $G_k$ of the form $Q\cup Z_\rho$, where $Q$ is a subset of $\Z[\frac1k]$ that is strictly confining under $\gamma$ with respect to some homomorphism $\rho\colon \Z^n\to \R$.  By Lemmas \ref{lem:rho=0}, \ref{lem:rho>0}, and \ref{lem:rho<0}, there are exactly $n+1$  such structures: $[S_1]=[Q_1\cup Z_{\rho_1}],\dots, [S_n]=[Q_n\cup Z_{\rho_n}],$ and $[S_-]=[Q_-\cup Z_{\rho_-}]$.  It remains to be shown that these structures  are pairwise incomparable.  By Theorem \ref{thm:main}, the Busemann pseudocharacters associated to the quasi-parabolic structures $[S_i]$ and  $[S_-]$ are proportional to $\rho^+_i$ and $\rho_-$, respectively.  If $i\neq j$, then $\rho_i(t_i)>0$ while $\rho_i(t_j)=0$, and $\rho_j(t_i)=0$ while $\rho_j(t_j)>0$.  Therefore $t_i$ is a loxodromic isometry in the structure $[S_i]$ and  elliptic in the structure $[S_j]$, while $t_j$ is elliptic  in the structure $[S_i]$ and  loxodromic in the structure $[S_j]$.  Therefore $[S_i]$ and $[S_j]$ are incomparable for all $i \neq j$.  
  
  We now show that for each $i=1,\dots, n$, the structures  $[S_i]$ and $[S_-]$ are incomparable. We consider the generator $t_i$ of $\Z^n$. By the definition of the Busemann pseudocharacter, the fixed point of $G_k$ on the Gromov boundary in the structure $[S_-]$ is the \emph{attracting} fixed point of $t_i$. In the structure $[S_i]$, the fixed point of $G_k$ on the Gromov boundary is the \emph{repelling} fixed point of $t_i$. If  $[S_-]$ and $ [S_i]$ were comparable then $G_k$ would fix both fixed points of $t_i$ in the action corresponding to the smaller structure. This contradicts that the structures are quasi-parabolic. Thus the proof is complete.
\end{proof}

\subsubsection{Geometry of the actions: Bass-Serre trees} \label{section:actiongeometry}
In this subsection and the next, we give explicit geometric descriptions of the quasi-parabolic actions associated to the structures described in the previous subsection.  

In this section, we consider the hyperbolic structure corresponding to \[Q_i=\Z\left[\frac{1}{p_1^{m_1} \cdots \widehat{p_i^{m_i}}\cdots p_n^{m_n}}\right]=\Z\left[\frac{1}{k_i}\right]\] and show that this is the equivalence class corresponding to  an action of $\Z[\frac{1}{k}]\rtimes_\gamma \Z^n$ on a certain Bass-Serre tree.

The group $G_k$ has as a subgroup \[H_i=\Z\left[\frac{1}{k_i}\right]\rtimes_\gamma \langle t_1, \ldots, \widehat{t_i}, \ldots, t_n\rangle\] since each generator $t_j$ for $j\neq i$ restricts to an automorphism of $\Z\left[1/k_i\right]$. Moreover,  the conjugation action of the generator $t_i$ on $\Z[\frac{1}{k}]$ restricts to an endomorphism from $\Z\left[1/k_i\right]$ onto the \emph{proper} additive subgroup $p_i^{m_i}\Z\left[1/k_i\right]$. Hence we may consider the ascending HNN extension of $H_i$ \[\langle H_i,s \mid shs^{-1}=t_i ht_i^{-1} \text{ for all } h\in H_i\rangle\]  (note that here the conjugate $t_iht_i^{-1}$ lies in $H_i$ so this presentation makes sense). The fundamental group of this HNN extension is isomorphic to $\Z[\frac{1}{k}]\rtimes_\gamma  \Z^n=G_k$ via the map which sends $s$ to $t_i$, and hence we have an action of $G_k$ on the corresponding Bass-Serre tree. Note in particular that $t_i$ acts loxodromically in this action whereas  $H_i$ acts elliptically. By \cite[Proposition 3.14]{AR}  the hyperbolic structure defined by this Bass-Serre tree is equal to $[Q_i \cup Z_{\rho^+_i}]$.

To give a more explicit description of the geometry of this action, we describe the Bass-Serre tree in another way. This construction should be compared to an analogous construction for $BS(1,k)$ in \cite[Section~3.1.3]{AR}.  The vertices are identified with the set $\Q_{p_i^{m_i}}\times \Z$ modulo the equivalence relation $(x,h)\sim (y,h)$ if $\|x-y\|\leq (p_i^{m_i})^{-h}=p_i^{-m_ih}$, where $\|\cdot \|$ denotes the $p_i^{m_i}$--adic absolute value. For any $x\in \Q_{p_i^{m_i}}$, there is an edge joining (the equivalence classes of) $(x,h)$ and $(x,h+1)$.  
The action of $\Z[\frac{1}{k}]\rtimes_\gamma \Z^n$ is defined via the following actions on vertices (here $a=1$ denotes a normal generator of $\Z[\frac{1}{k}]$):\[\begin{tabular}{l} $a: (x,h)\mapsto (x+1,h)$ \\
 $t_i: (x,h)\mapsto (p_i^{m_i}x,h+1)$  \\
 $t_j: (x,h)\mapsto (p_j^{m_j}x,h)$ for $j\neq i$. \end{tabular}\] To verify that this tree equipped with this action of $G_k$ is the same as the Bass-Serre tree we described, we need to study stabilizers of edges and vertices in the action. Before we do this, we give a particular example for concreteness.

\begin{ex}
We choose $k=6$ and the hyperbolic structure corresponding to $Q_2=\Z[\frac{1}{2}]$. Here the vertices of the Bass-Serre tree are identified with equivalence classes of pairs in $\Q_2\times \Z$. We denote our group by $G_6=\Z[\frac{1}{6}]\rtimes_\gamma \Z^2= \llangle a  \rrangle \rtimes_\gamma \langle s,t\rangle$. Here $a$ denotes the normal generator 1 of $\Z[\frac{1}{6}]$ and $\gamma(s)$ and $\gamma(t)$ act by multiplication by 2 and 3, respectively.

The generator $a$ is elliptic and acts as the 2--adic odometer $x\mapsto x+1$ on $\Q_2$; see the left hand side of Figure \ref{fig:bstrees}. The generator $s$ acts loxodromically and simply shifts vertices directly ``upward.'' The generator $t$ acts elliptically but has a more complicated action given by $x\mapsto 3x$ on $\Q_2$; see the right hand side of Figure~\ref{fig:bstrees}.  
\end{ex}

\begin{center}
\begin{figure}[h!]
\begin{tabular}{ll}
\includegraphics[width=0.5\textwidth]{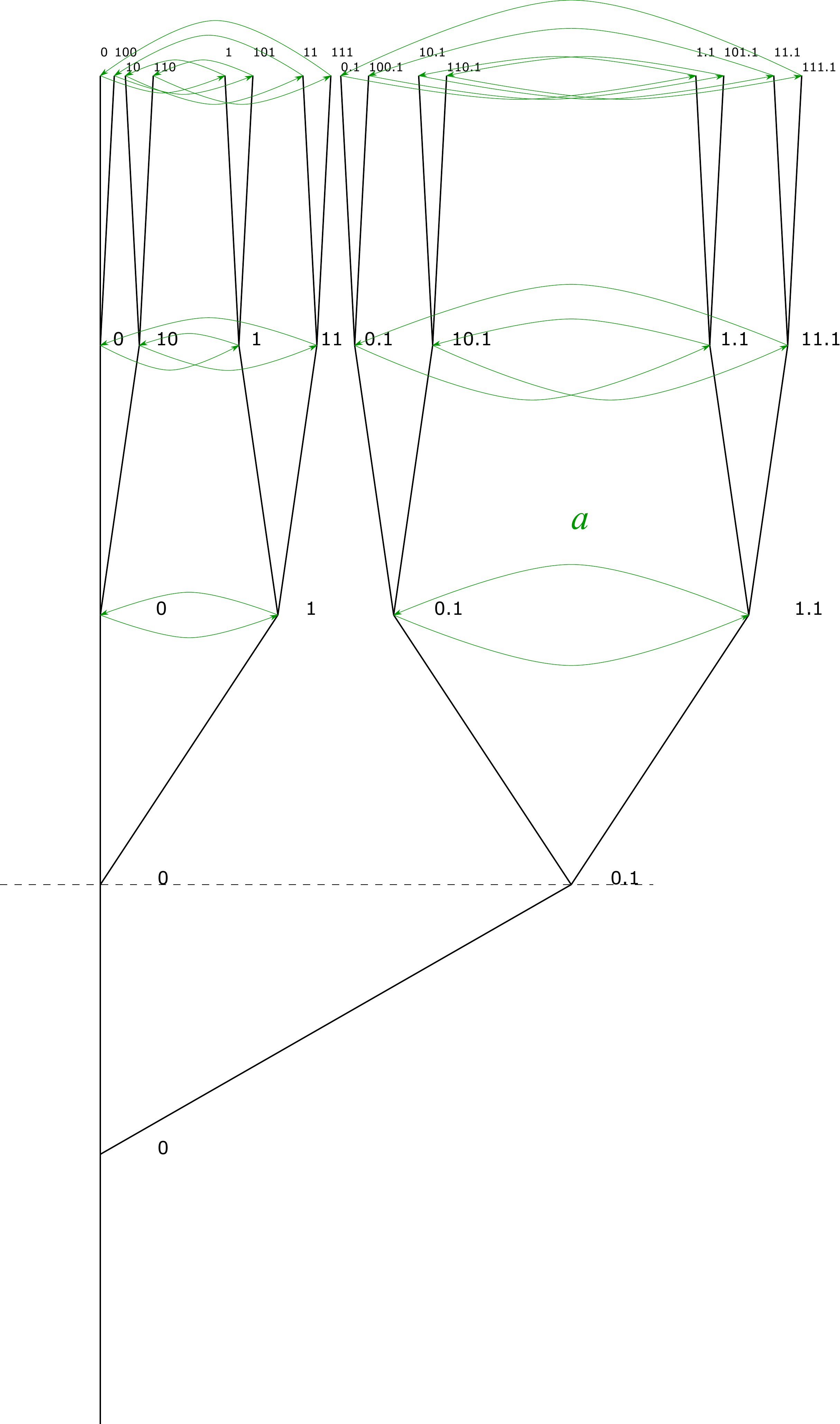}
&
\includegraphics[width=0.5\textwidth]{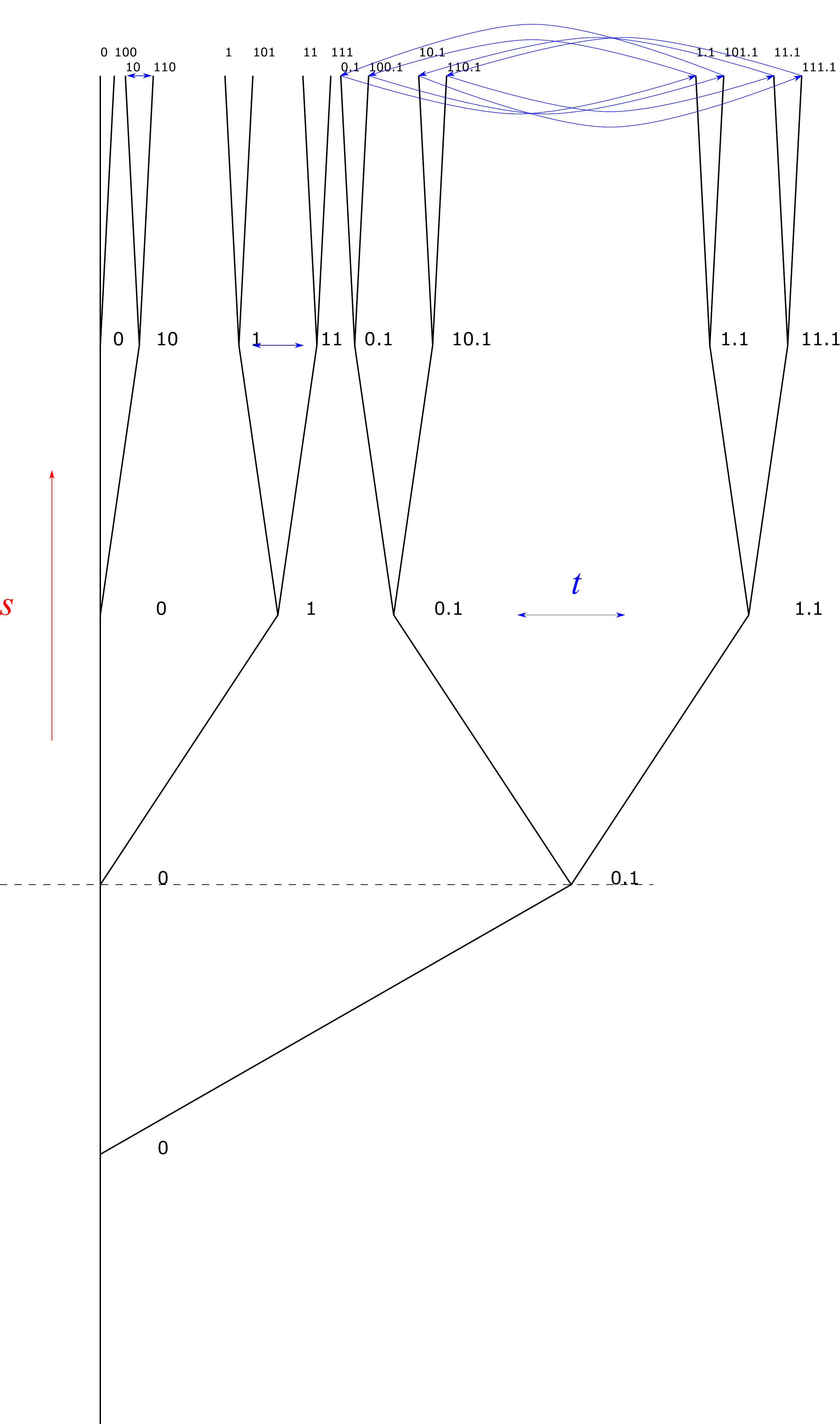}
\end{tabular}
\caption{Left: the action of the generator $a$ of $\Z[\frac{1}{6}]\rtimes \Z^2$. Right: the actions of the generators $s$ and $t$.  In both figures heights are implicit, so that a single 2--adic number is given for the label of each vertex. The vertices at height zero are demarcated by a dotted line.}
\label{fig:bstrees}
\end{figure}
\end{center}

Finally we show that the tree described above is equivariantly isomorphic to the Bass-Serre tree of $G_k$ corresponding to the ascending HNN extension with vertex group $\Z\left[1/k_i\right]\rtimes_\gamma \langle t_1,\ldots, \widehat{t_i},\ldots,t_n\rangle$. For simplicity we assume that $i=n$. Since the tree described above has a single orbit of vertices and a single orbit of edges, it suffices to describe the stabilizers of an edge and its two incident vertices. We focus on the edge $E$ between the vertices $v_0=(0,0)$ and $v_1=(0,1)$. Consider an element $g=xt_1^{r_1}\ldots t_n^{r_n}$, where $x\in \Z[\frac{1}{k}]$, which changes heights in the tree by $r_n$.  If $g$ fixes either vertex of $E$, then we must have $r_n=0$ and $g\in \Z[\frac{1}{k}] \rtimes \langle t_1,\ldots,t_{n-1}\rangle$. Thus \[g (0,0)=(x,0) \text{ and } g (0,1)=(x,1).\] Hence $g$ fixes $v_0$ if and only if the $p_n^{m_n}$--adic absolute value $\|x\|_{p_n^{m_n}}$ is at most one. This occurs exactly when $x$ is an element of  \[\Z\left[\frac{1}{k_n}\right]=\Z\left[\frac{1}{p_1^{m_1}\cdots p_{n-1}^{m_{n-1}}}\right].\] Similarly, $g$ fixes $v_1$ if and only if $\|x\|_{p_n^{m_n}}\leq p_n^{-m_n}$, which occurs exactly when $x\in p_n^{m_n}\Z\left[1/k_n\right]$. Thus the quotient graph of groups has vertex group $\Z\left[1/k_n\right]\rtimes \Z^{n-1}$. The edge group is also $\Z\left[1/k_n\right]\rtimes \Z^{n-1}$, and it embeds isomorphically onto the vertex group on one end and as the subgroup  $p_n^{m_n}\Z\left[1/k_n\right]\rtimes \Z^{n-1}$ on the other end. This completes the proof.

\subsubsection{Geometry of the actions: the hyperbolic plane}

The group $G_k$ also admits an action on the hyperbolic plane $\H^2$. We show that this action corresponds to the confining subset $Q_-$ associated to the homomorphism $\rho_-\colon\Z^n\to \R$ defined by $ t_j\mapsto -m_j \log(p_j)$.

To define the action $G_k\curvearrowright \H^2$, we consider the upper half plane model. We denote by $a$ the normal generator 1 of $\Z[\frac1k]$. For a point $w$ in the upper half plane we define \[a\colon w\mapsto w+1 \text{ and } t_j\colon w\mapsto p_j^{m_j}w.\] It is straightforward to verify that this induces an isometric action of $G_k$.  (To avoid confusion, we use $\cdot$ to denote this action.) One may check that for $r \in \Z[\frac1k]$ we have $r\cdot w = w+r$ and for $z=t_1^{a_1}\ldots t_n^{a_n} \in \Z^n$ we have $z\cdot w = p_1^{m_1a_1}\cdots p_n^{m_na_n}w$. 

We sketch the proof that the action $G_k\curvearrowright \H^2$ is equivalent to the action $G_k\curvearrowright \Gamma(G_k,Q_-\cup Z_{\rho_-})$ associated to the confining subset $Q_-$. This follows the proof of \cite[Proposition 3.16]{AR} closely, and we refer the interested reader there for more details.

The proof begins by invoking the Schwarz-Milnor Lemma (see Lemma~\ref{lem:MS}). We choose $i\in \H^2$ as a basepoint. The reader may check that the orbit of $i$ is dense in the horocycle $\Im(w)=1$. Applying powers of $t_1\ldots t_n\in \Z^n$ (corresponding to multiplication by $k$), we see that the orbit of $i$ is in fact dense in each horocycle $\Im(w)=k^j$, for $j\in \Z$. These horocycles are spaced a distance of $\log(k)$ apart, so any point of $\H^2$ is at distance at most $\log(k)$ from the orbit of $i$. In particular, the orbit of the ball of radius $\log(k)$ based at $i$ under $G_k$ covers $\H^2$. The Schwarz-Milnor Lemma now implies that the action $G_k\curvearrowright \H^2$ is equivalent to the action $G_k\curvearrowright \Gamma(G_k,S)$, where \[S=\{g\in G_k \mid d_{\H^2}(i,g\cdot i)\leq 2\log(k)+1\}.\] It remains to show that we have $[S] =[Q_-\cup Z_{\rho_-}]$ as generating sets of $G_k$.

First, we show that elements of $S$ have bounded word length with respect to $Q_-\cup Z_{\rho_-}$. We may write an element $g$ of $S$ as $g=rz$ with $r\in \Z[\frac1k]$ and $z=t_1^{a_1}\ldots t_n^{a_n} \in \Z^n$. We have $g\cdot i=li+r$ where $l=p_1^{m_1a_1}\cdots p_n^{m_na_n}$. Hence the distance $d_{\H^2}(i,g\cdot i)$ is at least $|\log(l)|=|\rho_-(z)|$. Since $d_{\H^2}(i,g \cdot i)$ is bounded by $2\log(k)+1$, this gives a bound on both $|\log(l)|$ and $|\rho_-(z)|$. The bound on $|\rho_-(z)|$ in turn implies a bound on the word length of $z$ with respect to $Z_{\rho_-}$. Similarly, using \[d_{\H^2}(i,g\cdot i)=2\operatorname{arcsinh}\left(\frac{1}{2}\sqrt{\frac{r^2+(l-1)^2}{l}}\right)\geq 2\operatorname{arcsinh}\left(\frac{1}{2}\frac{r}{\sqrt{l}}\right)\] and the bound on $l$ just obtained, we obtain a bound on $|r|$. Thus, choosing any $z'\in \Z^n$ with $\rho_-(z')$ larger than a constant depending only on $2\log(k)+1$, we have $|\gamma(z')(r)|<1$. In particular, $\gamma(z')(r)\in Q_-$. Since the word length of $z'$ with respect to $Z_{\rho_-}$ is bounded, this gives a bound on the word length of $r=(z')^{-1}\gamma(z')(r)(z')$ with respect to $Q_-\cup \Z_{\rho_-}$.

Finally, we show that elements of $Q_-\cup Z_{\rho_-}$ have bounded word length with respect to $S$. This follows by reversing the argument of the last paragraph. An element of $Q_-$ moves $i$ by distance at most $1$ and hence lies in $S$ by definition. On the other hand, if $z=t_1^{a_1}\ldots t_n^{a_n} \in Z_{\rho_-}$ and $l=p_1^{a_1m_1}\cdots p_n^{a_nm_n}$ then we have $ z\cdot i = li$ and the bound on $|\rho_-(z)|=|-\log(l)|$ implies a bound on $d_{\H^2}(i,z\cdot i)$.  Using the $\log(k)$--density of the orbit of $i$ under $G_k$, we obtain a bound on the word length of $z$ with respect to $S$. This completes the proof.

\subsection{Proof of Theorem \ref{thm:Z1k}}
Proposition \ref{prop:qp} gives a complete description of the quasi-parabolic structures $\mc H_{qp}(G_k)$.  We now consider the other structures in turn, beginning with lineal structures.  We will show that the lineal structures are in bijection with equivalence classes of homomorphisms $\rho\colon \Z^n\to \R$.

Fix $\rho\colon \Z^n\to \R$.  Then $\Z[\frac1k]$ is confining under $\gamma$ with respect to $\rho$, and by Theorem \ref{thm:main}, $[\Z[\frac1k]\cup Z_\rho]$ is a lineal structure.  This defines a map $\phi$  from the set of equivalences classes of homomorphims $\Z^n\to \R$ to the set of lineal structures, where $\phi([\rho])= [\Z[\frac1k]\cup Z_\rho]$.  We will show that $\phi$ is a bijection.

Let $\rho'\colon \Z^n\to\R$ be another homomorphism, and suppose first that $\rho\sim\rho'$.  Then since $\rho$ and $\rho'$ are proportional, every element of $Z_\rho$ has bounded word length with respect to $Z_{\rho'}$, and vice versa.  Therefore $[\Z[\frac{1}{k}]\cup Z_\rho]=[\Z[\frac{1}{k}]\cup Z_{\rho'}]$, which shows that $\phi$ is well-defined.

Suppose next that $\rho\not\sim \rho'$. Then the kernels of $\rho$ and $\rho'$ do not coincide, and there are elements of $\Z^n$ which are a bounded distance from the kernel of $\rho$ while being an unbounded distance from the kernel of $\rho'$. More precisely, there exists a constant $B$ and a sequence $\{z_i\}_{i=1}^\infty \subseteq \Z^n$ with $\rho(z_i)\geq i$ and $|\rho'(z_i)|\leq B$ for all $i$. Since $\rho$ and $\rho'$ are quasi-isometries from $\Gamma(G_k, \Z[\frac1k]\cup Z_\rho)$ to $\R$ and from $\Gamma(G_k, \Z[\frac1k]\cup Z_{\rho'})$ to $\R$, respectively, this proves that $[\Z[\frac1k]\cup Z_{\rho'}]\neq[\Z[\frac1k]\cup Z_\rho]$.  This shows that $\phi$ is injective.
 Moreover, by \cite[Theorem~4.22]{ABO}, we can also conclude that  $[\Z[\frac1k]\cup Z_\rho]$ and  $[\Z[\frac1k]\cup Z_{\rho'}]$ are incomparable.  

Finally, we will show that $\phi$ is surjective.  Given any lineal structure  $[S]$ on $G_k$, Theorem \ref{thm:actiontoconf} implies that $[S]=[\Z[\frac{1}{k}]\cup Z_\rho]$ for some associated homomorphism $\rho\colon \Z^n\to\R$. 
Therefore, $\phi$ is surjective and so a bijection.

Every quasi-parabolic structure dominates the lineal structure defined by its Busemann pseudocharacter.  In particular, $[\Z[\frac1k]\cup Z_{\rho^+_i}]\preceq[Q_i\cup Z_{\rho^+_i}]$ and $[\Z[\frac1k]\cup Z_{\rho_-}]\preceq[Q_-\cup Z_{\rho_-}]$.  Moreover by an analogous argument to the above proof that $\phi$ is injective, we see that if $\rho\not\sim\rho^+_i$ or $\rho\not\sim\rho_-$, then $[\Z[\frac1k]\cup Z_{\rho}]$ is not dominated by $[Q_i\cup Z_{\rho^+_i}]$ or $[Q_-\cup Z_{\rho_-}]$, respectively.

For our choice of $k \geq 2$, the group $G_k=\Z[\frac1k]\rtimes \Z^n$ is solvable, and so contains no free subgroups.  By the ping-pong lemma, we must have $\mathcal H_{gt}(G_k)=\emptyset$.

Finally, for any group $G$, $\mathcal H_e(G)$ consists of a single element which is the smallest element of $\mathcal H(G)$.  This completes the proof of the theorem.

\subsection{Computation of Bieri-Neumann-Strebel invariants}
\label{section:BNS}

In this section, we compute the Bieri-Neumann-Strebel (BNS) invariants (or simply Bieri-Strebel invariants, in this case) $\Sigma(G_k)$ of the generalized solvable Baumslag-Solitar groups $G_k$. This invariant was first computed for $G_k$ by Sgobbi and Wong   using different methods \cite{SgobbiWong}. For the definitions and background   on Bieri-Strebel and Bieri-Neumann-Strebel invariants, we refer the reader to \cite{BieriStrebel2,BieriNeumannStrebel}.

Our computation is an almost immediate corollary of Theorem~\ref{thm:Z1k} and a result of Brown characterizing the (complement of the) BNS invariant of a group in terms of certain  actions of the group on trees  (\cite[Corollary 7.4]{Brown}, given below as Theorem~\ref{thm:browns}).  Essentially, computing the \textit{entire} poset of hyperbolic structures of a group also gives a description of all of the actions of the group on trees, which then yields the BNS invariant of the group.  We now briefly introduce the terminology necessary to describe Brown's characterization. 

An $\R$-tree $X$  with an action of a group $G$ has an associated hyperbolic length function $\ell$, and the action $G\curvearrowright X$ is called \emph{abelian} if there is a homomorphism $\chi\colon G\to \R$ with $\ell = |\chi|$. Thus $\ell$ uniquely determines $\chi$ up to multiplication by $\pm 1$. There is a unique fixed end $e$ of $X$, and we choose a representative $\chi_X$ for $\pm \chi$ with the property that $\chi_X(g) = \ell(g)$ if $g$ translates away from $e$ and $\chi_X(g) = -\ell(g)$ otherwise. If $G$ has no invariant point or line in $X$, then the action is called \emph{non-trivial}. Note that if $G$ is solvable, then $\chi_X$ is exactly the Busemann pseudocharacter for $G\curvearrowright X$. Recall that when $G$ is finitely generated,  there is a sphere $S(G)$ of homomorphisms $\rho\colon G\to \R$ considered up to scaling by numbers in $\R_{>0}$; we denote the equivalence class of $\rho$ by $[\rho]$. The BNS invariant $\Sigma(G)$ is a subset of $S(G)$. Brown's characterization of $\Sigma(G)$ is:

\begin{thm}[{\cite[Corollary 7.4]{Brown}}]
\label{thm:browns}
If $G$ is finitely generated, then the complement $\Sigma(G)^c$ in $S(G)$ is the set of equivalence classes $[\chi_X]$ obtained as above, where $X$ is an $\R$-tree with a non-trivial, abelian $G$-action.
\end{thm}

The computation of the BNS invariant for $G_k$ now follows  quickly from Theorems~\ref{thm:Z1k} and \ref{thm:browns}.  Recall that $n$ is the number of distinct prime factors of $k$.

\begin{cor}
The Bieri-Neumann-Strebel invariant $\Sigma(G_k)$ is the complement of $\{[\rho_1^+],\ldots,[\rho_n^+]\}$ in $S(G_k)\cong S^{n-1}$.
\end{cor}

\begin{proof}
Suppose that $\chi_X$ is as in Theorem \ref{thm:browns}. If the action $G_k \curvearrowright X$ is not cobounded, then we may consider the minimal invariant sub-tree of $X$, and the resulting homomorphism $\chi_X$ (i.e., the Busemann pseudocharacter) remains unchanged, while the action on it is now cobounded. Hence we may assume without loss of generality that $G_k\curvearrowright X$ itself is cobounded. Since this is a cobounded action of $G_k$ on a hyperbolic space, there is a coarsely equivariant quasi-isometry from $X$ to one of the spaces $\ast, \R, T_i,$ or $\H^2$ described in Section \ref{section:QPofGk}. Since $X$ is a non-trivial tree, the quasi-isometry must be to one of the trees $T_i$. Using the coarse equivariance and considering whether elements translate towards or away from the fixed ends in $T_i$ and $X$, we see that for all $g\in G_k$,  $\chi_X(g)>0$ if and only if $\rho_i^+(g) > 0$. By Lemma \ref{lem:multiplekernel}, $\rho_i^+ \sim \chi_X$. Thus $\Sigma(G)^c = \{[\rho_1^+],\ldots,[\rho_n^+]\}$.
\end{proof}


\bibliographystyle{abbrv}
\bibliography{researchbib.bib}

\begin{thebibliography}{10}

\bibitem{ABO}
C.~R. Abbott, S.~Balasubramanya, and D.~Osin.
\newblock Hyperbolic structures on groups.
\newblock {\em Algebr. Geom. Topol.}, 19(4):1747--1835, 2019.

\bibitem{AR}
C.~R. Abbott and A.~J. Rasmussen.
\newblock Actions of solvable {B}aumslag-{S}olitar groups on hyperbolic metric
  spaces.
\newblock {\em arXiv:1906.04227}, 2019.

\bibitem{Largest}
C.~R. Abbott and A.~J. Rasmussen.
\newblock Largest hyperbolic actions and quasi-parabolic actions of groups.
\newblock {\em arXiv:1910.14157}, 2019.

\bibitem{Bal}
S.~H. Balasubramanya.
\newblock Hyperbolic structures on wreath products.
\newblock {\em J. Group Theory}, 23(2):357--383, 2020.

\bibitem{BieriNeumannStrebel}
R.~Bieri, W.~D. Neumann, and R.~Strebel.
\newblock A geometric invariant of discrete groups.
\newblock {\em Invent. Math.}, 90(3):451--477, 1987.

\bibitem{BieriStrebel2}
R.~Bieri and R.~Strebel.
\newblock Valuations and finitely presented metabelian groups.
\newblock {\em Proc. London Math. Soc. (3)}, 41(3):439--464, 1980.

\bibitem{BMNR}
C.~Bonatti, I.~Monteverde, A.~Navas, and C.~Rivas.
\newblock Rigidity for {$C^1$} actions on the interval arising from
  hyperbolicity {I}: solvable groups.
\newblock {\em Math. Z.}, 286(3-4):919--949, 2017.

\bibitem{Brown}
K.~S. Brown.
\newblock Trees, valuations, and the {B}ieri-{N}eumann-{S}trebel invariant.
\newblock {\em Invent. Math.}, 90(3):479--504, 1987.

\bibitem{BBRT}
J.~Brum, N.~M. Bon, C.~Rivas, and M.~Triestino.
\newblock Solvable groups and affine actions on the line.
\newblock {\em arXiv:2209.00091}, 2022.

\bibitem{scl}
D.~Calegari.
\newblock {\em scl}, volume~20 of {\em MSJ Memoirs}.
\newblock Mathematical Society of Japan, Tokyo, 2009.

\bibitem{Amen}
P.-E. Caprace, Y.~Cornulier, N.~Monod, and R.~Tessera.
\newblock Amenable hyperbolic groups.
\newblock {\em J. Eur. Math. Soc. (JEMS)}, 17(11):2903--2947, 2015.

\bibitem{Gro}
M.~Gromov.
\newblock {\em Hyperbolic Groups}, pages 75--263.
\newblock Springer New York, New York, NY, 1987.

\bibitem{GrovesManning}
D.~Groves and J.~F. Manning.
\newblock Dehn filling in relatively hyperbolic groups.
\newblock {\em Israel J. Math.}, 168:317--429, 2008.

\bibitem{boundaries}
I.~Kapovich and N.~Benakli.
\newblock Boundaries of hyperbolic groups.
\newblock In {\em Combinatorial and geometric group theory ({N}ew {Y}ork,
  2000/{H}oboken, {NJ}, 2001)}, volume 296 of {\em Contemp. Math.}, pages
  39--93. Amer. Math. Soc., Providence, RI, 2002.

\bibitem{Man}
J.~F. Manning.
\newblock Actions of certain arithmetic groups on {G}romov hyperbolic spaces.
\newblock {\em Algebr. Geom. Topol.}, 8(3):1371--1402, 2008.

\bibitem{NeumannShapiro}
W.~Neumann and M.~Shapiro.
\newblock A short course in geometric group theory, notes for the anu workshop
  january/february 1996.
\newblock Available at
  \url{https://www.math.columbia.edu/~neumann/preprints/canberra.ps}, 1996.

\bibitem{Papa}
P.~Papasoglu.
\newblock Strongly geodesically automatic groups are hyperbolic.
\newblock {\em Invent. Math.}, 121(2):323--334, 1995.

\bibitem{Plante}
J.~F. Plante.
\newblock Solvable groups acting on the line.
\newblock {\em Trans. Amer. Math. Soc.}, 278(1):401--414, 1983.

\bibitem{SSV}
W.~Sgobbi, D.~C. Silva, and D.~Vendr\'{u}sculo.
\newblock The {$R_\infty$} property for nilpotent quotients of generalized
  solvable {Baumslag-Solitar} groups.
\newblock {\em arXiv:2208.02647}, 2022.

\bibitem{SgobbiWong}
W.~Sgobbi and P.~Wong.
\newblock The {BNS} invariants of the generalized solvable {Baumslag-Solitar}
  groups and of their finite index subgroups.
\newblock {\em arXiv:2110.14834}, 2021.

\bibitem{TabackWhyte}
J.~Taback and K.~Whyte.
\newblock The large-scale geometry of some metabelian groups.
\newblock {\em Michigan Math. J.}, 52(1):205--218, 2004.

\bibitem{TabackWong}
J.~Taback and P.~Wong.
\newblock Twisted conjugacy and quasi-isometry invariance for generalized
  solvable {B}aumslag-{S}olitar groups.
\newblock {\em J. Lond. Math. Soc. (2)}, 75(3):705--717, 2007.

\end{thebibliography}
 
\end{document}